%% file: main.tex
\documentclass[onefignum,onetabnum,dvipsnames]{siamart171218}

\input{ex_shared}

\ifpdf
\hypersetup{
  pdftitle={Time-Domain BIE in Thermoelasticity: mixed boundary conditions},
  pdfauthor={G. C. Hsiao, and T. S\'{a}nchez-Vizuet}
}
\fi

\begin{document}
\maketitle

\begin{abstract}
We study boundary integral formulations for an interior/exterior initial boundary value problem arising from the thermo-elasto-dynamic equations in a homogeneous and isotropic domain. The time dependence is handled, based on Lubich's approach, through a passage to the Laplace domain. We focus on the cases where one of the unknown fields satisfies a Dirichlet boundary condition, while the other one is subject to conditions of Neumann type. In the Laplace domain, combined single- and double-layer potential boundary integral operators are introduced and proven to be coercive. Based on the Laplace domain estimates, it is possible to prove the existence and uniqueness of solutions in the time domain. This analysis complements previous results that may serve as the mathematical foundation for discretization schemes based on the combined use of the boundary element method and convolution quadrature. 
\end{abstract}

\begin{keywords}
Time-domain boundary integral equations, Linear thermoelasticity, Boundary integral operators, Fundamental solution .
\end{keywords}

\begin{AMS}
74F05,  74J05,  45P05, 35L05,  65N38, 35J20.
\end{AMS}

%
\section{Introduction}
%
Boundary integral formulations are a well established tool for dealing with stationary and time-harmonic boundary value problems for linear elliptic partial differential equations originating from models of mathematical physics and mechanics. One of the practical advantages of these formulations is the reduction of dimensionality achieved by exploiting Green's second formula and the availability of a Green's function associated to the linear partial differential operator with constant coefficients. Thus reducing the problem into one posed solely on the boundary of the original domain of definition. Moreover, when the original problem is posed in an unbounded domain exterior to a bounded boundary, a boundary integral formulation presents itself as a very attractive alternative amenable for numerical computations.  

However, for the non-stationary case boundary integral formulations had not received much attention until recent years; with much progress been made on their analysis and discretization within the current century. In this work, which concludes the analysis started in \cite{HsSa:2020}, we follow the approach set forth by Lubich \cite{Lu:1988a, LuSc:1992}, and Sayas and co-workers \cite{LaSa:2009a,Sayas:2016}, based on the seminal articles by Bamberger and Ha-Duong \cite{BaHa:1986a, BaHa:1986b}. The central idea of this analysis technique is to transform the problem into the Laplace domain, where Green's functions are readily available and easier to analyze, establish the coercivity of the boundary integral operators involved, and then transfer these results into the time domain. 

This communication is concerned with the application of Time Domain Boundary Integral Methods (TDBIMs) to a non-stationary initial boundary value problem  coming from linear thermoelasticity. It concludes the analysis started in \cite{HsSa:2020}, where we focused on ``pure'' Dirichlet or Neumann boundary conditions. Here, on the contrary, we consider the case of combined boundary conditions 
, i.e. where one of the two unknowns satisfies Dirichlet boundary conditions, while the other one is subject to Neumann type conditions.

We start in \cref{sec:GoverningEquations} by briefly introducing the equations governing the evolution of the displacement and temperature fields in linear thermoelasticity; both in the time domain and the corresponding  transformed formulations in the Laplace domain. The necessary background required to introduce the boundary integral formulation, the layer potential ansatz and the thermoelastic boundary integral operators are presented in \cref{sec:Preliminaries}. Most of the analysis is contained in \cref{sec:CombinedBVP}, where the Laplace-transformed problem is studied and proven to be uniquely solvable. The stability bounds obtained in this section in terms of the Laplace parameter, are then translated into time-domain results and estimates in the final \cref{sec:TimeDomain}.

%
\section{Governing equations}\label{sec:GoverningEquations}
%
\subsection{The system in the time domain}\label{sec:TimeDomainSystem}
Let $\Omega^-$ be a bounded domain in $\mathbb{R}^d$ (for $d = 2,3$) with Lipschitz boundary $\Gamma$ and let $\Omega^+ := \mathbb{R}^d \setminus \overline\Omega^-$ be its exterior. We will consider either of the interior or exterior domains to be occupied by a homogeneous, isotropic, linear thermoelastic medium with constant density $\rho$. The elastic properties of the medium are characterized by the Lam\'e parameters $\lambda$ and $\mu$, while the thermal properties are determined by the thermal diffusivity coefficient $\kappa=k/\delta$, given in terms of the thermal diffusivity $k$ and the specific heat $\delta$. 

We are interested in the evolution of the elastic displacement field $\bf {U}$ and temperature variation field $\Theta$ in the aforementioned medium over the time interval $[0,T]$. At the initial time, when the solid is at rest and unperturbed, the distribution of temperatures is given by $\Theta_0$ however, in a thermoelastic medium, elastic stresses give rise to temperature variations and temperature variations produce mechanical stress on the body. In the linear regime, this effect is described by the Duhamel-Neumann law \cite{Duhamel:1837,Maugin:2014} connecting the stress and temperature through the \textit{thermoelastic stress} $\boldsymbol \sigma(\mathbf U,\Theta)$ and the \textit{thermoelastic heat flux} $\mathbf F(\mathbf U,\theta)$
 \[
 \boldsymbol \sigma(\mathbf U, \Theta) := \widetilde{\boldsymbol \sigma} (\mathbf U) - \gamma \, \Theta\, \mathbf I \quad \mbox{and }\quad
\boldsymbol{\mathcal F}(\mathbf U,\Theta):= - \eta\frac{\partial}{\partial t}\mathbf U + \kappa\nabla\Theta.
\]
Here, the identity operator is denoted by $\mathbf I$ and the coupling constants $\gamma$ and $\eta$ are defined by
\[
\gamma := (\lambda + 2/3 \,\mu) \alpha, \qquad \text{ and } \qquad \eta := \gamma \Theta_0 /k,
\]
where $\alpha$ is the volumetric coefficient of thermal expansion, and the quantity $\lambda + (2/3)\,\mu$ is known as the \textit{bulk modulus} of the solid. The purely elastic linearized stress and strain tensors $\widetilde{\boldsymbol {\sigma}}(\mathbf U)$ and $\widetilde{\boldsymbol \varepsilon} (\mathbf U) $  are given as usual by 
\[
 \widetilde{\bm \sigma} ({\bf U}) = (\lambda\;\nabla\cdot{\bf U}) {\bf I} + 2 \mu \; \widetilde{\bm \varepsilon} ({\bf U}) \quad \mbox{and }\quad
\quad \widetilde{\bm \varepsilon}({\bf U}) = \tfrac{1}{2}( \nabla{\bf U} + (\nabla{\bf U})^\top).
\]
In the thermoelastic medium, the given physical constants $  \rho, \lambda, \mu, \gamma, \eta,
\kappa$, are assumed to satisfy the inequalities:
\[
 \rho > 0, \;\;\mu > 0 , \;\;\lambda + 2/3 \mu > 0, \;\;\gamma/\eta > 0, \;\;\kappa > 0.
\]
Under these physically meaningful assumptions, and defining the Lam\'e operator $\Delta^*$ by 
\[
 \Delta^* \mathbf{U} := \mu \Delta \mathbf{U} + (\lambda + \mu) \nabla\,\nabla\cdot \mathbf{U} = \nabla\cdot \widetilde{\bf \bm{\sigma}}(\bf U),
 \]
 the conservation of energy and momentum give rise to the following governing thermo-elasto-dynamic equations   
\begin{subequations}\label{eq:GoverningEquations}
\begin{alignat}{6}
\label{eq:2.1}
\rho \frac{\partial^2\mathbf{U}} {\partial t^2} - \Delta^{*} \mathbf{U} + \gamma \, \nabla~\Theta 
 =\,& \mathbf{0} \qquad& \text{ in } \Omega^\mp\times(0,T), \\ 
\label{eq:2.2}
 \frac{1}{\kappa}\frac{\partial \Theta}{\partial t} - \Delta \Theta + \eta\; \frac{\partial}{\partial t}(\nabla\cdot \mathbf{U}) =\,& 0 \qquad& \text{ in } \Omega^\mp\times(0,T).
 \end{alignat}
which are complemented with the causal initial conditions
\begin{equation}
\label{eq:InitialConditions}
 {\bf U}(x,t) = {\bf 0}, ~ \frac{\partial}{\partial t}{\bf U}(x,t )= {\bf 0}, \quad \mbox{and} \quad {\Theta}(x,t) = 0, 
\end{equation}
for $-\infty < t \leq 0, x \in \Omega^\mp$, together with boundary conditions. In the previous work \cite{HsSa:2020}, we considered the problem where both temperature and displacement field are subject to
\begin{description}
\item (a) Dirichlet boundary conditions: 
\[
{\bf U}(x,t) = {\bf F}(x,t), \quad \mbox{and} \quad \Theta(x,t) = F(x,t)\quad \mbox{ on } \quad \Gamma_T := \Gamma \times (0, T\,], 
\]
\item (b) Neumann boundary conditions:
\[
\boldsymbol{\sigma} ( {\bf U}, \Theta ) {\bf n} = {\bf G}(x,t) \quad\mbox{ and}\quad 
\nabla\,\Theta\cdot\mathbf n  = G(x,t) \quad \mbox{on}\quad \Gamma_T.
\]
\end{description}
for ${\bf F}(x,t)$, $F(x,t)$, ${\bf G}(x,t)$ and $G(x,t)$ are given smooth functions.

In the current work, we will focus on the cases where the boundary conditions are \textit{combined}, meaning that one of the fields is subject to a Dirichlet condition while the other one is subject to a Neumann condition. More precisely, we will be analyzing the cases
\begin{description}
\item (c) Dirichlet-Neumann boundary conditions: 
\begin{equation}\label{eq:DN-Condition}
{\bf U}(x,t) = {\bf F}(x,t), \quad \mbox{and} \quad \nabla\Theta(x,t)\cdot\mathbf n = G(x,t)\quad \mbox{ on } \quad \Gamma_T, 
\end{equation}
\item (d) Neumann-Dirichlet boundary conditions:
\begin{equation}\label{eq:ND-Condition}
\boldsymbol{\sigma} ( {\bf U}, \Theta ) {\bf n} = {\bf G}(x,t) \quad\mbox{ and }\quad \Theta  = F(x,t) \quad \mbox{on}\quad \Gamma_T.
\end{equation}
\end{description}
\end{subequations}
Since the operators involved are linear, it is natural to construct solutions of these problems by using boundary integral methods. However, such an approach requires the availability of a fundamental solution for the time-dependent 
equations (\ref{eq:2.1})  and  (\ref{eq:2.2}). However, time dependent fundamental solutions may not be available in general or may be considerably more difficult to construct and to handle analytically and computationally than those of time-independent equations. This will lead us to consider the techniques set forth by Lubich \cite{Lu:1988a,Lu:1988,Lu:1994} and Sayas \cite{LaSa:2009a, Sayas:2016} which sidestep the need for a time dependent fundamental solution and instead proceed by transforming the system into the Laplace domain, where the analysis is carried out before translating the results obtained back into the time domain. This technique requires carefully studying the dependence of the stability and coercivity bounds for the operators appearing in the Laplace domain in terms of the complex Laplace parameter.
%
\subsection{The system in the Laplace domain}\label{sec:LaplaceDomainSystem}
%
Before transforming the problem \cref{eq:GoverningEquations} into the Laplace domain, we must introduce some notation. Throughout the paper, we will say that a vector-valued function $ ({\bf u}, \theta)^{\top} $ defined over a domain $\mathcal O$ is \textit{regular} if $ ({\bf u}, \theta)^{\top} \in C^2(\mathcal O)^4\cap C^1(\overline{\mathcal O})^4$ (also denoted by $ {\bf C}^2(\mathcal O)\cap {\bf C}^1(\overline{\mathcal O})$).

The complex plane will be denoted by $\mathbb{C}$; for the remainder of this article, it should be understood that the Laplace parameter $s$ belongs to the \textit{positive half plane}, i.e.
\[
s \in \mathbb{C}_+ := \{ z \in \mathbb{C} : \mathrm{Re}(z) > 0\}.
\]
Provided it exists, the Laplace transform for an ordinary, causal complex-valued function $F \colon [0, \infty) \to \mathbb{C}$ will be denoted by 
\[
f(s)= \mathcal{L}{F}(s) := \int_0^\infty e^{-st} F(t) dt.
\]
In what follows, we shall adhere to the convention of denoting time-domain functions by a capital letter and their Laplace transforms by the corresponding lower case letter. To make the notation less cumbersome, we will suppress the explicit dependence of the variables with respect to both position and the Laplace parameter. Hence, we will let $\mathbf{u} := \mathbf{u}(x,s)= \mathcal{L}\{ {\bf U}(x,t)\}$ and $ \theta:=\theta(x,s) = \mathcal{L}\{{\Theta(x,t)} \}$. 

Upon Laplace-transformation, equations \cref{eq:2.1} and \cref{eq:2.2} become
\begin{subequations}\label{eq:ThermoElasticLaplaceSystem}
\begin{alignat}{6} 
 \Delta^{*} \mathbf{u} -  \rho s^2 \mathbf{u} - \gamma\, \nabla ~\theta =\,& \mathbf{0} &\quad& \mbox{in}\quad \Omega^\mp, \\
 \Delta \theta - (s/\kappa) ~\theta - s \eta ~\nabla\cdot \mathbf{u} =\,& 0  &\quad& \mbox{in} \quad \Omega^\mp. 
\end{alignat}
\end{subequations}
The system of equations \cref{eq:ThermoElasticLaplaceSystem} can be written in matrix-vector form as
\begin{equation} \label{eq:AdjointThermoealsticMatrix}
{\bf B}(\partial_x , s) \begin{pmatrix}
{\mathbf u} \\
{\theta}\\
\end{pmatrix}
:= \left (
\begin{array} {ll}
  \Delta^* - \rho s^2 & -\gamma \,\nabla \\
  -s \,\eta\, \nabla^{\top} & \Delta -s/\kappa \\
 \end{array} \right )
 \begin{pmatrix}
{\mathbf u} \\
\theta 
\end{pmatrix}
=  \begin{pmatrix} {\bf 0}\\ 0 \end{pmatrix}\quad \mbox{in} \quad \Omega^\mp. 
\end{equation}
Following Kupradze \cite{Ku:1979}, the non-self adjoint matrix-operator ${\bf B}(\partial_x , s)$ will be referred to as the {\it thermoelastic pseudo-oscillation operator}; its adjoint operator ${\bf B}^*(\partial_x , s)$ may be obtained from ${\bf B}(\partial_x , s)$ by replacing $\gamma$ with $ - s\eta $ and vice versa, namely
\begin{equation}\label{eq:AdjointThermoealsticMatrix}
\mathbf{B}^*(\partial_x , s) \begin{pmatrix}
{\boldsymbol v} \\
{ v}\\
\end{pmatrix}
:= \left (
\begin{array} {ll}
  \Delta^* - \rho s^2 & s\,\eta \;\nabla \\
  \gamma~ \; \nabla^{\top} & \Delta - s/\kappa \\
 \end{array} \right )
 \begin{pmatrix}
{\boldsymbol v} \\
 v 
\end{pmatrix}.
\end{equation}
The following result (a Laplace-Domain counterpart of Kupradze's \cite[p.139 ]{Ku:1979} and Cakoni's \cite{Cakoni:2000}) splits the regular solution to the thermoelastic oscillation equations \cref{eq:ThermoElasticLaplaceSystem} into three wave-like functions. We  refer the reader to \cite{HsSa:2020} for the proof.
\begin{lemma}
The regular solution $ ({\bf u}, \theta)$ of \cref{eq:AdjointThermoealsticMatrix} admits in the domain of regularity a representation 
of the form:
\begin{equation}\label{eq:2.6}
 (\mathbf{u}, \theta) = (\mathbf{u}_1, \theta_1) + (\mathbf{u}_2, \theta_2) + (\mathbf{u}_3, \theta_3), 
 \end{equation}
with $({\bf u}_k, \theta_k), k = 1,2,3 $ satisfying 
\begin{alignat*}{8}
(\Delta -\lambda^2_1)\mathbf{u}_1 =\,& {\bf 0}, \qquad& (\Delta -\lambda^2_2)\mathbf{u}_2 =\,& {\bf 0}, \qquad& (\Delta -\lambda^2_3) \mathbf{u}_3 =\,& {\bf 0}, \\
\nabla \times \mathbf{u}_1 =\,& {\bf 0}, \qquad& \nabla\times\mathbf{u}_2 =\,& {\bf 0}, \qquad& \nabla\cdot \mathbf{u}_3 =\,& 0, \\
(\Delta -\lambda^2_1)\theta_1 =\,& 0, \qquad& (\Delta -\lambda^2_2)\theta_2 =\,& 0, \qquad& \theta_3 =\,& 0.
\end{alignat*} 
The constants $\lambda^2_1, \lambda^2_2 $ are determined by the equations:
\begin{equation}
\lambda^2_1 + \lambda^2_2  = \frac{s}{\kappa} ( 1 + \epsilon) + \lambda^2_p, \qquad \qquad \lambda^2_1 \lambda^2_2 \;= \left(\frac{s}{\kappa}\right) \lambda^2_p , \label{eq:2.7a} 
\end{equation} 
with $ \lambda^2_3 = \rho s^2/ \mu, \; \lambda^2_p = \rho s^2/ (\lambda + 2 \mu)$ and $\epsilon = \gamma\, \eta\, \kappa / (\lambda + 2 \mu)$.
\end{lemma}
Just like in elastodynamics, $ c_s:= \sqrt{\mu/\rho}$ and $ c_p:= \sqrt{(\lambda + 2 \mu)/\rho}$
are the phase velocities for the transversal and longitudinal wave, respectively. For most of bodies, the 
parameter $\epsilon \in (0,1) $ is much smaller than unity---see, e.g. \cite{No:1975}. When $\epsilon = 0 $, the deformation and temperature fields decouple. In this case, the wave numbers may be deduced explicitly from eq.\cref{eq:2.7a} 
\[
\lambda^2_1 = \frac{s}{ \kappa}, \qquad \lambda^2_2 =\frac{ \rho\, s^2}{\lambda + 2 \mu}, \qquad \mbox{ and } \qquad \lambda^2_3 = \frac{\rho\, s^2} {\mu}.
\]
%
\section{Preliminaries}\label{sec:Preliminaries}
%
\subsection{Notation}
%
Over either of the open domains $\Omega^\mp$, the standard $\mathrm L^2(\Omega^\mp)$ inner products for scalars, vectors and matrices will be denoted respectively by
\[
(u,v)_{\Omega^\mp} := \int_{\Omega^\mp} uv \;, \quad (\boldsymbol u,\boldsymbol v)_{\Omega^\mp} := \int_{\Omega^\mp} \boldsymbol u \cdot \boldsymbol v \;, \quad (\boldsymbol U,\boldsymbol V)_{\Omega^\mp} := \int_{\Omega^\mp} \boldsymbol U :\boldsymbol V \;,
\]
where $\boldsymbol U :\boldsymbol V$ is the Frobenius inner product for matrices. These inner products induce a natural norm that will be denoted by $\|\cdot\|_{\Omega^\mp}$. We will use the brackets $\langle\cdot,\cdot\rangle_{\Gamma}$ for the $L_2(\Gamma)$ inner products on the boundary $\Gamma := \partial \Omega^\mp$. Inner products for the restrictions of scalars, vectors and matrices to the boundary are defined in a similar fashion. Forms and products will always be assumed to be \textit{bilinear} and complex conjugation will be made explicitly when required.

If a function $u$, defined in either $\Omega^\mp$, is continuously extendible at a point $y \in \Gamma$, then $u^-(y)$ and $u^+(y)$ (alternatively $\{u(y)\}^-$ and $\{u(y)\}^+$) will denote the one sided limits
\[
u^-(y) \equiv \left\{ u(y)\right\}^-:= \lim_{ \Omega^- \ni x \to y \in \Gamma}\, u(x), \;\; \mbox{and} \;\; u^+(y) \equiv \left\{ u(y)\right\}^+:= \lim_{ \Omega^+ \ni x \to y \in \Gamma}\, u(x).
\]
The same notation will be used for one sided limits for vector-valued and matrix-valued extendible functions, where the limits are taken component by component. The \textit{trace} operator, or restriction to the boundary will be denoted by the symbol $\Big |_{\Gamma}$, in order to keep notation as light as possible, its use will be omitted whenever a restriction to the boundary is clear from the context. 

We will rely heavily on standard notation and terminology of Sobolev space theory. Vector-valued spaces will be denoted in boldface, and should be understood as copies of their scalar counterparts therefore, for $d=2,3$. 
\[
\mathbf L^2(\Omega^\mp):= \left(\mathrm L^2(\Omega^\mp)\right)^d, \quad  \mathbf H^1(\Omega^\mp):= \left(\mathrm H^1(\Omega^\mp)\right)^d, \quad \mathbf H^{1/2}(\Gamma):= \left(\mathrm H^{1/2}(\Gamma)\right)^d.
\]
We also adopt the following conventions: for any function $u$ defined on both sides of $\Gamma$ its jump across the boundary $\Gamma$ will be denoted by
\[
\jump{u} := ( u^- -  u^+),
\]
while the average of its traces from inside and outside of $\Gamma$ will be denoted by
\[
\ave{u} := ( u^- + u^+ ) /2.
\]
Just as before, for vector  and matrix valued functions, the jump and average are defined component wise in a similar fashion.

\subsection{Boundary conditions}\label{sec:BoundaryConditions}
%
Four basic types of boundary conditions may be imposed on the Laplace-domain system \cref{eq:ThermoElasticLaplaceSystem}. Dirichlet and Neumann boundary conditions  were treated by the authors in \cite{HsSa:2020}. In the current communication,  we will focus on the remaining  two combined Dirichlet and Neumann type of boundary conditions.

In what follows let  $\mathbf{T}: \mathbf H^1(\Omega^\mp) \rightarrow \mathbf H^{-1/2}(\Gamma)$  be the elastic traction operator defined by
\[
\mathbf{T}\, \mathbf{u} : = \widetilde{\bm \sigma} (\mathbf{u}) \;\mathbf{n} = \lambda (\nabla\cdot\mathbf{u} )\mathbf{n} + 2 \mu \frac{\partial \mathbf{u}} {\partial n} + \mu \mathbf{n} \times \nabla \mathbf{u}. 
\]
Using the traction operator, we can define the following operators mapping $\mathbf H^1(\Omega^\mp)\times \mathrm H^1(\Omega^\mp)$ to spaces defined over the boundary, each of them suited for different kinds of boundary conditions.

\begin{description}
\item{\bf I.} Dirichlet boundary conditions. 
\begin{align*}
\mathcal{R}_D :  \mathbf H^1(\Omega^\mp)\times \mathrm H^1(\Omega^\mp) &\longrightarrow  \mathbf H^{1/2}(\Gamma)\times \mathrm H^{1/2}(\Gamma) \\
\left(\mathbf{u}, \theta\right) &\longmapsto \left ( \mathbf{u}, \theta\right)\Big |_{\Gamma}
\end{align*}
\item {{\bf II.}} Neumann boundary conditions. 
\begin{align*} 
\mathcal{R}_N :  \mathbf H^1(\Omega^\mp)\times \mathrm H^1(\Omega^\mp) &\longrightarrow  \mathbf H^{-1/2}(\Gamma)\times \mathrm H^{-1/2}(\Gamma) \\
\left(\mathbf{u}, \theta\right) &\longmapsto \left ( \mathbf{T}\,\mathbf{u} -\gamma\,\theta\, \mathbf{n}\,,\, \partial_n\theta\right )
\end{align*}
\item{{\bf III.}} Dirichlet-Neumann combined boundary conditions.
\begin{align*}
\mathcal{R}_{DN} :  \mathbf H^1(\Omega^\mp)\times \mathrm H^1(\Omega^\mp) &\longrightarrow  \mathbf H^{1/2}(\Gamma)\times \mathrm H^{-1/2}(\Gamma) \\
\left(\mathbf{u}, \theta\right) &\longmapsto \left (\mathbf{u}\Big |_{\Gamma}, \partial_n\,\theta\right) 
\end{align*}
\item{{\bf IV.}} Neumann-Dirichlet combined boundary conditions.
\begin{align}
\nonumber
 \mathcal{R}_{N D} :  \mathbf H^1(\Omega^\mp)\times \mathrm H^1(\Omega^\mp) &\longrightarrow  \mathbf H^{-1/2}(\Gamma)\times \mathrm H^{1/2}(\Gamma) \\
\label{eq:RND}
\left(\mathbf{u}, \theta\right) &\longmapsto \left ( \mathbf{T}\,\mathbf{u} -\gamma\,\theta\, \mathbf{n}\,,\, -\theta \Big |_{\Gamma}\right).
\end{align}
\end{description}
The operators $\mathcal{R}^*_D, \mathcal{R}^*_N, \mathcal{R}^*_{DN},$ and $\mathcal{R}^*_{ND}$---adjoint to the ones defined above---are defined from the previous definitions by the substitution $\gamma \mapsto -\eta\,s$.
 
\textbf{Remark:} Whenever any of the operators defined above acts on functions of two variables, an additional subscript will be added to indicate the relevant variable. Therefore, in an expression of the type $\mathcal R_{N_y} \boldsymbol F(x,y)$, the normal vector $\mathbf n$ should be understood as being a function of $y$, and differentiation should be performed with respect to $y$. This will become specially relevant in \cref{sec:GreenFormulas} and \cref{sec:LayerPotentials}, when layer potentials and boundary integral operators involving the fundamental solution of system \cref{eq:ThermoElasticLaplaceSystem} are defined.
%
\subsection{Associated Bilinear Forms}
%
For pairs $(\boldsymbol{v},  v)$ and $ (\mathbf{u}, \theta)$ of regular functions belonging to $\mathbf{C}^2(\Omega^\mp) \cup \mathbf{C}^1(\bar\Omega^\mp)$, we will define the following bilinear operators (complex conjugation will always be done explicitly) associated to the thermoelastic oscillation problem.
\begin{align}
\nonumber
\mathcal{A}_{\Omega^\mp} \left( \left( \mathbf{u}, \theta\right) , \left(\boldsymbol{v}, v\right)\right) &:=
 \int_{\Omega^\mp}\Big( \widetilde{\bm{\sigma}}( \mathbf{u} ) : \widetilde{\bm{ \varepsilon }}( \boldsymbol{ v}) + \rho s^2 \mathbf{u} \cdot \boldsymbol{v} - \gamma\,\theta\nabla\,\cdot\boldsymbol{v} \\
\label{eq:bilinearA}
&  \quad \qquad + \,s\eta \nabla\cdot\mathbf{u}\,  v + \nabla \theta \cdot \nabla  v + (s/\kappa) \theta\,  v \Big), \\
\nonumber
\mathcal{B}_{\Omega^\mp}\left( \left( \mathbf{u}, \theta\right) , \left(\boldsymbol{v}, v\right)\right) &:= \int_{\Omega^\mp} \left(\Delta^{*} \mathbf{u} -  \rho s^2 \mathbf{u} - \gamma\, \nabla ~\theta\right)\cdot\boldsymbol{v} + \left(\Delta \theta - \left(s/\kappa \right) ~\theta - s \eta ~\nabla\cdot \mathbf{u} \right)\,v \\
\label{eq:bilinearB}
& = \int_{\Omega^\mp} (\boldsymbol v, v)\,{\bf B}(\partial_x , s)\,(\mathbf u, \theta)^\top, \\
\nonumber
\mathcal{B}^*_{\Omega^\mp}\left( \left(\boldsymbol{v}, v\right), \left( \mathbf{u}, \theta\right) \right) &:= \int_{\Omega^\mp} \left(\Delta^{*} \boldsymbol{v} -  \rho s^2 \boldsymbol{v} +s\eta\, \nabla ~v\right)\cdot\mathbf{u} + \left(\Delta v - \left(s/\kappa \right) ~v + \gamma ~\nabla\cdot \boldsymbol{v} \right)\,\theta \\
\label{eq:bilinearB*}
& = \int_{\Omega^\mp} (\mathbf u, \theta)\,{\bf B}^*(\partial_x , s)\,(\boldsymbol v, v)^\top.
\end{align}
Note that $\mathcal{B}_{\Omega^\mp}$ (resp. $\mathcal{B}^*_{\Omega^\mp}$) arises from testing the strong form of the system \cref{eq:ThermoElasticLaplaceSystem} (resp. the adjoint system) with  $(\boldsymbol v, v)$ (resp. with $(\mathbf u, \theta)$), while $\mathcal{A}_{\Omega^\mp}$ arises from the weak formulation of  \cref{eq:ThermoElasticLaplaceSystem}. 
%
\subsection{Green Formulas}\label{sec:GreenFormulas}
%
Using the notation defined above, the first Green's formula takes the form: 
\begin{equation} \label{eq:FirstGreenFormula}
\mathcal{A}_{\Omega^\mp}\left( \left(\mathbf{u},\theta\right),\left(\boldsymbol{v},v\right)\right) + \mathcal{B}_{\Omega^\mp} \left( \left(\mathbf{u},\theta\right),\left(\boldsymbol{v},v\right)\right)
 = \pm \int_{\Gamma}\mathcal{R}_N\left(\mathbf{u},\theta\right)\,\cdot\left(\boldsymbol{v},v\right)
\end{equation}
while the second Green formula is given by 
\begin{align}\label{eq:SecondGreenFormula}
\mathcal{B}_{\Omega^\mp}\left( \left(\mathbf{u},\theta\right),\left(\boldsymbol{v},v\right)\right)
- \mathcal{B}^*_{\Omega^\mp}\left(\left(\boldsymbol{v},v\right),\left(\mathbf{u},\theta\right)\right)
&= \pm \int_{\Gamma} \left( \mathcal{R}_N\left(\mathbf{u},\theta\right)\cdot\left(\boldsymbol{v},v\right) - \mathcal{R}^*_N\left(\boldsymbol{v},v\right)\cdot\left(\mathbf{u},\theta\right)\right).
\end{align}
If $(\boldsymbol {v}, v)$ is replaced by the fundamental solution $\underline {\underline{\bf E}}(x,y;s)$ for the adjoint operator $\mathbf{B}^*(\partial_y, s)$, the second Green formula \cref{eq:SecondGreenFormula} provides the following integral representation for the solution of \cref{eq:AdjointThermoealsticMatrix}: 
\begin{align} \label{eq:RepresentationFormula}
\left( \mathbf {u}, \theta\right) (x)
& = \pm\int_{\Gamma} \left( \underline {\underline{\bf E}}^{\top}(x,y;s)\cdot {\mathcal R}_N \left( \mathbf {u}, \theta\right) -  \mathcal R^*_{N_y} ~\underline {\underline{\bf E}}^{\top}(x,y;s)\cdot  \left( \mathbf {u}, \theta\right) 
\right) d_y\Gamma , \quad x \in \Omega^\mp.
\end{align}
The explicit expression for the two and three dimensional fundamental solution derivation for the fundamental solution $\underline{\underline{\bf {E}}}(x,y;s)$ can be found in the \cref{sec:FundamentalSolutions}. For a detailed derivation the reader is referred to \cite[Appendix A]{HsSa:2020}.
%
\subsection{Thermoelastic layer potentials and boundary integral operators}\label{sec:LayerPotentials}
%
The integral representation \cref{eq:RepresentationFormula} suggests looking for solutions of the form
\[
\left( \mathbf {u}, \theta\right) (x)
 = \pm\int_{\Gamma} \left( \underline {\underline{\bf E}}^{\top}(x,y;s)\cdot \left( \boldsymbol {\lambda}, \varsigma\right) -  \mathcal R^*_{N_y} ~\underline {\underline{\bf E}}^{\top}(x,y;s)\cdot  \left( \boldsymbol {\phi}, \varphi\right) 
\right) d_y\Gamma , \quad x \in \Omega^\mp,
\]
where $\varphi , {\bm \phi}, \varsigma,$ and ${\bm \lambda}$ are unknown \textit{density functions} defined over $\Gamma$ and satisfying certain regularity properties to be specified later. This process is known as \textit{direct method} and, when dealing with standard Dirichlet or Neumann boundary value problems, has the property that the densities coincide respectively with the trace of the temperature field, the trace of the elastic displacement, the normal heat flux and the normal elastic traction.  

Each of the terms in \cref{eq:RepresentationFormula}, taken separately, also define functions that satisfy the system \cref{eq:ThermoElasticLaplaceSystem}. They are known as the \textit{simple layer potential} (or single layer potential)  $\mathcal S(s)$ and \textit{double layer potential} $\mathcal D(s)$
\begin{alignat*}{6}
 \mathcal{S} (s)\left(\boldsymbol{\lambda},\varsigma\right)^{\top}(x) &:= 
 \int_{\Gamma} \underline {\underline{\bf E} } (x,y;s)\left(\boldsymbol{\lambda},\varsigma\right)^{\top} (y)\, d_y{\Gamma} &\qquad x \not \in \Gamma,\\
 \mathcal{D} (s)\left(\boldsymbol{\phi},\varphi\right)^{\top} (x) &:= 
 \int_{\Gamma} \left(\mathcal{R}^{*}_{N_y} ~\underline {\underline{\bf E}}^{\top}(x,y;s)\right)^{\top}
\left(\boldsymbol{\phi},\varphi\right)^{\top}(y) \,d_y{\Gamma} &\qquad x \not \in \Gamma.
\end{alignat*}
These operators are well suited to be used as ansatz for thermoelastic problems where both the temperature and displacement fields satisfy either Dirichlet or Neumann boundary conditions. To handle cases where one of the unknowns satisfies a Dirichlet-type condition and the other one a Numann-type condition we define the following layer potentials:
\begin{alignat*}{6}
 \mathcal{Q}_{SD}(s)~\left(\boldsymbol{\lambda},\varphi\right)^{\top} (x) &:=
 \int_{\Gamma} \left(\mathcal{R}^{*}_{DN_y} \underline {\underline{\bf E}}^{\top}(x,y;s)\right)^{\top} \left(\boldsymbol{\lambda},\varphi\right)^{\top}(y)\, d_y{\Gamma} \, &\qquad x \not \in \Gamma, \\
 \mathcal{Q}_{DS}(s)\left(\boldsymbol{\phi},\varsigma\right)^{\top} (x) &:=
 \int_{\Gamma} \left(\mathcal{R}^{*}_{ND_y} \underline {\underline{\bf E}}^{\top}(x,y;s)\right)^{\top}\left(\boldsymbol{\phi},\varsigma\right)^{\top} (y)\, d_y{\Gamma} \, &\qquad x \not \in \Gamma.
\end{alignat*}
The definition of the layer potentials can be made more symmetric by noting that, since $\mathcal{R}^{*}_{D_y} = \mathcal{R}_{D} = \mathbf{I}_{d+1} $, the simple-layer potential can also be expressed as 
\begin{align*}
 \mathcal{S} (s)\left(\boldsymbol{\lambda},\varsigma\right)^{\top}(x) &:=
 \int_{\Gamma}\left( \mathcal{R}^{*}_{D_y}\underline {\underline{\bf E}}^{\top}(x,y;s)\right)^{\top}\left(\boldsymbol{\lambda},\varsigma\right) (y)\, d_y{\Gamma}, \, \qquad x \not \in \Gamma. 
\end{align*}
These four thermoelastic layer potentials give rise to the operators that will be used to reformulate our problem as a system of boundary integral equations. In the sequel, for simplicity ${\top}$ superscripts applied to pairs of density functions will be suppressed when there is no confusion in the context.   For density functions $\boldsymbol \phi \in \mathbf H^{1/2}(\Gamma)$, $\varphi \in \mathrm H^{1/2}(\Gamma)$, $\boldsymbol \lambda \in \mathbf H^{-1/2}(\Gamma)$, and $\varsigma \in \mathrm H^{-1/2}(\Gamma)$, we define the following boundary integral operators:

\begin{description}
\item{\bf I.} Single-layer operator.
\begin{align*}
\mathcal V(s) :\, \mathbf H^{-1/2}(\Gamma) \times \mathrm H^{-1/2}(\Gamma) &\longrightarrow  \mathbf H^{1/2}(\Gamma) \times \mathrm H^{1/2}(\Gamma) \\
 (\boldsymbol \lambda, \varsigma) &\longmapsto \{\mathcal R_D\mathcal S(s) (\boldsymbol \lambda, \varsigma)\}^{\mp}.
\end{align*} 
\item{\bf II.} Transposed Double-layer operator.
\begin{align*}
\mathcal K^\prime(s) :\, \mathbf H^{-1/2}(\Gamma) \times \mathrm H^{-1/2}(\Gamma) &\longrightarrow  \mathbf H^{-1/2}(\Gamma) \times \mathrm H^{-1/2}(\Gamma) \\
(\boldsymbol \lambda, \varsigma) &\longmapsto \ave{\mathcal R_N \mathcal S(s) (\boldsymbol \lambda, \varsigma)}.
\end{align*}
\item{\bf III.} Double-layer operator.
\begin{align*}
\mathcal K(s) :\, \mathbf H^{1/2}(\Gamma) \times \mathrm H^{1/2}(\Gamma) &\longrightarrow  \mathbf H^{1/2}(\Gamma) \times \mathrm H^{1/2}(\Gamma) \\
(\boldsymbol \phi, \varphi) &\longmapsto \ave{\mathcal R_D\mathcal D(s) (\boldsymbol \phi, \varphi) }.
\end{align*}
\item{\bf IV.} Hypersingular operator.
\begin{align*}
\mathcal W(s) :\, \mathbf H^{1/2}(\Gamma) \times \mathrm H^{1/2}(\Gamma) &\longrightarrow  \mathbf H^{-1/2}(\Gamma) \times \mathrm H^{-1/2}(\Gamma) \\
(\boldsymbol \phi, \varphi) &\longmapsto - \{\mathcal R_N\mathcal D(s) (\boldsymbol \phi, \varphi) \}^{\mp}.
\end{align*}
\item{\bf V.} Single-Double layer combined operator.
\begin{align*}
\mathcal C_{SD}(s) :\, \mathbf H^{-1/2}(\Gamma) \times \mathrm H^{1/2}(\Gamma) &\longrightarrow  \mathbf H^{1/2}(\Gamma) \times \mathrm H^{-1/2}(\Gamma) \\
(\boldsymbol \lambda, \varphi) &\longmapsto \{\mathcal R_{DN}\mathcal Q_{\mathcal {SD}}(s) (\boldsymbol \lambda, \varphi)\}^{\mp}.
\end{align*}
\item{\bf VI.} Double-Single layer combined operator.
\begin{align*}
\mathcal C_{DS}(s) :\, \mathbf H^{1/2}(\Gamma) \times \mathrm H^{-1/2}(\Gamma) &\longrightarrow  \mathbf H^{-1/2}(\Gamma) \times \mathrm H^{1/2}(\Gamma) \\
(\boldsymbol \phi, \varsigma) &\longmapsto \{\mathcal R_{ND}\mathcal Q_{\mathcal{DS}}(s) (\boldsymbol \phi, \varsigma) \}^{\mp}.
\end{align*}
\end{description}
The single-layer and hypersingular operators appear when using the representation formula \cref{eq:RepresentationFormula} as ansatz (i.e. the direct method). The double layer operator and its transpose appear when using the indirect method to study ``pure" Dirichlet or Neumann boundary value problems; they have been studied in \cite{HsSa:2020}. The last two operators will appear when studying problems in which one of the unknowns satisfies Dirichlet boundary conditions and the other one Neumann boundary conditions; they are the subject of the present article and will be studied in detail in the next section.
%
\section{ Combined Boundary Value Problems}\label{sec:CombinedBVP} 
%
\subsection{Energy norm in the Laplace domain}
%
 For the remainder of the article, it will be convenient to use the following notation for the real part of the Laplace parameter
\[
\sigma : = \mathrm{Re}(s) \qquad\text{ and }\qquad \underline{\sigma}:= \min\{1, \mathrm{Re}(s)\}.
\]
The following $s$-dependent energy norms are well suited for the Laplace domain analysis (see, e.g. \cite{HsSa:2020, HsSaSa:2016, HsSaSaWe:2016}) will be used in what follows. 
\begin{alignat*}{6}
\triple{\mathbf u}_{|s|, \Omega^\mp}^2 &:= \left( \widetilde{\bm{ \sigma}} ( {\mathbf u}), \overline{\widetilde{\bm{\varepsilon}} (\mathbf {u}} ) \right)_{\Omega^\mp} +  \rho \| s \; \mathbf u \|^2_{\Omega^\mp} \quad& \mathbf u &\in {\mathbf H}^1(\Omega^\mp), \\
 \triple{\theta}^2_{|s|, \Omega^\mp} &:= \| \nabla \theta\|^2_{\Omega^\mp} + \kappa^{-1}
 \| \sqrt{ |s|} \; \theta \|_{\Omega^\mp}^2 \quad& \theta &\in H^1(\Omega^\mp), \\
 \triple{ ({\bf u}, \theta)}^2_{|s|, \Omega^\mp} &:= \triple{\mathbf u}^2_{|s|, \Omega^\mp} + \triple{\theta}^2_{|s|, \Omega^\mp}
 \quad & ({\bf u}, \theta) &\in {\mathbf H}^1(\Omega^\mp)\times H^1(\Omega^\mp). 
 \end{alignat*}
We note that these norms are equivalent for any value of $s$, as established by the following relations
\begin{subequations}\label{eq:4.6}
\begin{align} 
\label{eq:4.6a}
\underline{\sigma }\triple{\mathbf{u}}_{1, \Omega^\mp} &\leq \triple{\mathbf{u}}_{|s|, \Omega^\mp} \leq
\frac{|s|}{\underline{\sigma}} \triple{\mathbf{u}}_{1, \Omega^\mp} ,\\
\label{eq:4.6b}
\sqrt{\underline{\sigma}} \triple{\theta}_{1, \Omega^\mp} &\leq \triple{\theta}_{|s|, \Omega^\mp} \leq 
\sqrt{\frac{|s|}{\underline {\sigma} } } \triple{\theta}_{1,\Omega^\mp},\\
 \label{eq:4.6c}
\underline{\sigma} \triple{({\bf u}, \theta)}_{1, \Omega^\mp} &\leq \triple{({\bf u}, \theta)}_{|s|, \Omega^\mp} \leq \frac{|s|}{{\underline{\sigma}}^{3/2}}\triple{({\bf u}, \theta)}_{1, \Omega^\mp} ,
 \end{align}
 \end{subequations}
which can be obtained from the inequalities: 
 \[
\underline{\sigma} \leq \min\{1, |s|\},\quad \mbox{and} \quad \underline{\sigma}\,\max\{1, |s|\}
 \leq |s|,~ \;\forall s \in \mathbb{C}_+.
 \]
We remark that for $s=1$, the norm $\triple{\theta}_{1, \Omega^-} $ is equivalent to $\|\theta\|_{H^1(\Omega^-)} $ and so is the energy norm $\triple{ \mathbf{u}}_{1, \Omega^-}$ equivalent to the $\mathbf{H}^1(\Omega^-)$-norm of ${ \mathbf{u}} $ by the second Korn inequality \cite{Fi:1972}. 
 %
 \subsection{Boundary integral formulation and transmission problems}
%
We are interested in the thermoelastic oscillation problem \cref{eq:ThermoElasticLaplaceSystem} with combined boundary conditions. In order to rigorously state the distributional version of the problem, we will first define the spaces
\begin{align*}
\mathbf H^1_{\Delta^*}(\Omega^\mp) &:= \{\mathbf u \in \mathbf H^1(\Omega^\mp): \Delta^*\mathbf u \in \mathbf L^2(\Omega^\mp)\}, \\
\mathrm H^1_{\Delta}(\Omega^\mp) &:= \{ u \in \mathbf H^1(\Omega^\mp): \Delta u \in \mathbf L^2(\Omega^\mp)\}.
\end{align*}
Given either $(\mathbf f, g) \in\mathbf H^{1/2}(\Gamma) \times \mathrm H^{-1/2}(\Gamma)$ for the combined Dirichlet-Neumann problem, or $ (\mathbf g, f) \in\mathbf H^{-1/2}(\Gamma) \times \mathrm H^{1/2}(\Gamma) $ for the combined Neumann-Dirichlet problem, we seek a pair $(\mathbf u, \theta)\in \mathbf H^1_{\Delta^*}(\Omega^\mp)\times \mathrm H^1_{\Delta}(\Omega^\mp)$ satisfying
\begin{subequations}\label{eq:CombinedProblems}
\begin{equation}\label{eq:CombinedProblemPDE}
\mathbf B (\partial_x,s) (\mathbf u, \theta) = \mathbf 0 \quad \text{ in } \quad \mathbf L^2(\Omega\mp) \times \mathrm L^2(\Omega\mp),
\end{equation}
and either set of boundary conditions:
\begin{alignat}{6}
\label{eq:CombinedDirichletNeumann}
\mathcal R_{DN}(\mathbf u, \theta) &= (\mathbf f, g) &&\; \text{ in } \mathbf H^{1/2}(\Gamma) \times \mathrm H^{-1/2}(\Gamma) &\quad \text{(Dirichlet-Neumann problem),} \\
\label{eq:CombinedNeumannDirichlet}
\mathcal R_{ND}(\mathbf u, \theta) &= (\mathbf g, f) &&\; \text{ in } \mathbf H^{-1/2}(\Gamma) \times \mathrm H^{1/2}(\Gamma) &\quad \text{(Neumann-Dirichlet problem).}
\end{alignat}
\end{subequations}
All the equalities above must be understood in the sense of distributions.

When looking for solutions the problems \cref{eq:CombinedProblems} either of the combined layer potentials $\mathcal Q_{SD}$ and $\mathcal Q_{DS}$ defined in \cref{sec:LayerPotentials} can be used as an ansatz since, as it turns out, any function defined in terms of them in fact satisfies the distributional PDE. In fact, direct---if tedious---computations show that the follwing results hold for slightly different transmission problems for the operator $\mathbf B(\partial_z,s)$.
\begin{proposition}[Combined transmission problems] \label{prop:CombinedTransmission}
\\ \textbf{Combined Dirichlet-Neumann problem.} Let $({\boldsymbol{\lambda}}, \varphi) \in \mathbf H^{-1/2}(\Gamma) \times \mathrm H^{1/2}(\Gamma) $ and consider the function
\[
(\mathbf u_{\boldsymbol\lambda}, \theta_{\varphi}) :=  \mathcal{Q}_{SD} (s)(\boldsymbol{\lambda}, 
 \varphi) \in {\bf H}^1(\mathbb{R}^d\setminus \Gamma) \times \mathrm H^1(\mathbb{R}^d\setminus \Gamma).
 \]
Then 
 \begin{subequations} \label{eq:6.3}
\begin{align}
\mathbf{B}(\partial_x , s) (\mathbf u_{\boldsymbol\lambda}, {\theta}_{\varphi}) 
& = {\bf 0} \qquad \mbox{in}\quad \mathbb{R}^d \setminus \Gamma, \label{eq:6.3a}\\
\jump{\mathcal{R}_{DN} (\mathbf u_{\boldsymbol\lambda}, \theta_{\varphi}) } & = {\bf 0},\label{eq:6.3b}\\
\jump{ \mathcal{R}_{ND} (\mathbf u_{\boldsymbol\lambda}, \theta_{\varphi})} & = (\boldsymbol{\lambda}, \varphi). \label{eq:6.3c}
\end{align}
\end{subequations}
\\ 
\textbf{Combined Neumann-Dirichlet problem.} Similarly, let $ ({\bm{\phi}}, \varsigma) \in {\bf H}^{1/2}(\Gamma) \times \mathrm H^{-1/2}(\Gamma) $ and consider the function
\[
(\mathbf u_{\boldsymbol\phi}, \theta_{\varsigma}) :=  \mathcal{Q}_{DS}(s)(\boldsymbol{\phi}, 
 \varsigma) \in {\bf H}^1(\mathbb{R}^d\setminus \Gamma) \times \mathrm H^1(\mathbb{R}^d\setminus \Gamma).
 \]
Then 
\begin{subequations} \label{eq:6.9}
\begin{align}
\mathbf{B}(\partial_x , s) (\mathbf u_{\boldsymbol\phi}, {\theta}_{\varsigma}) 
& = {\bf 0}\qquad \mbox{in}\quad \mathbb{R}^d \setminus \Gamma, \label{eq:6.9a}\\
\jump{\mathcal{R}_{DN}(\mathbf u_{\boldsymbol\phi}, \theta_{\varsigma})} & =(\boldsymbol{\phi}, \varsigma) , \label{eq:6.9b}\\
\jump{\mathcal{R}_{ND} (\mathbf u_{\boldsymbol\phi}, \theta_{\varsigma})} & = {\bf 0}. \label{eq:6.9c}
\end{align}
\end{subequations}
\end{proposition} 
\noindent
Since functions defined in terms of the layer potentials applied to arbitrary pairs $({\boldsymbol{\lambda}}, \varphi) \in {\bf H}^{-1/2}(\Gamma) \times \mathrm H^{1/2}(\Gamma) $  or $ ({\boldsymbol{\phi}}, \varsigma) \in {\bf H}^{1/2}(\Gamma) \times \mathrm H^{-1/2}(\Gamma) $ satisfy the equations \cref{eq:6.3a} and \cref{eq:6.9a} , it only remains to find the appropriate density functions that would satisfy the corresponding boundary conditions of the problem. This process results in a system of boundary integral equations for the unknown densities.

For the combined Dirichlet-Neumann boundary value problem \cref{eq:CombinedProblemPDE} and \cref{eq:CombinedDirichletNeumann}, if we seek a solution in the form of 
\begin{equation}\label{eq:DirichletNeumannRepresentation}
( \mathbf{u}_{\boldsymbol{\lambda}}, \theta_{\varphi}) = 
 \mathcal Q_{SD}(s) \left(\boldsymbol{\lambda},\varphi\right) (x),
\end{equation}
then the boundary condition leads to an integral equation of the first kind for the unknown density  functions 
$ (\boldsymbol{\lambda}, \varphi) $ such that 
\begin{equation}\label{eq:6.5}
\mathcal{R}_{DN}\,\mathcal{Q}_{SD}(s)\left(\boldsymbol{\lambda},\varphi\right) (x) = \mathcal C_{SD} (s)\left(\boldsymbol{\lambda},\varphi\right) = \left( \mathbf{f},g\right), 
\end{equation}
where $({\bf f}, g) \in \mathbf H^{1/2}(\Gamma) \times \mathrm H^{-1/2}(\Gamma) $ is the given boundary  data. In the same manner, for the the combined Neumann-Dirichlet boundary value problem, we propose a solution of the form
\begin{equation}\label{eq:NeumannDirichletRepresentation}
( \mathbf{u}_{\boldsymbol{\phi}}, \theta_{\varsigma}) = 
 \mathcal{Q}_{DS}(s)\left(\boldsymbol{\phi},\varsigma\right) (x),
\end{equation}
then, imposing the boundary condition leads to the following integral equation of the first kind for the unknown density functions $(\bm{\phi}, \varsigma)$ 
 \begin{equation}\label{eq:6.11}
\mathcal{R}_{ND}\,\mathcal{Q}_{DS}(s) \left(\boldsymbol{\phi},\varsigma\right) (x) = \mathcal C_{DS}(s) \left(\boldsymbol{\phi},\varsigma\right) = \left( \mathbf{ g},f\right), 
\end{equation}
where $({\bf g}, f) \in \mathbf {H}^{-1/2}(\Gamma) \times \mathrm H^{1/2}(\Gamma) $ is the given boundary  data. Reciprocally, the solutions to the boundary integral equations above can be recovered from those of the boundary value problems. In view of the equivalence of these two sets of problems (boundary integral formulation vs. boundary value problem), it is enough to prove the existence and uniqueness of one of the two formulations. %
\subsection{Existence and Uniqueness}\label{sec:ExistenceAndUniqueness}
%

Note that---unlike the solution to the PDE formulation of the thermoelastic problem that is posed only either in the interior domain $\Omega^-$ or in the exterior domain $\Omega^+$---the layer potential ansatz \cref{eq:DirichletNeumannRepresentation} and \cref{eq:NeumannDirichletRepresentation} defined using the solutions of the boundary integral equations \cref{eq:6.5} and \cref{eq:6.11} as densities are defined over the \textit{entire} space $\mathbb R^d\setminus\Gamma$. Therefore, in order to make use of \cref{prop:CombinedTransmission}, we must work with the slightly more general transmission problems. Note that an interior (resp. exterior) boundary value problem can be converted into a transmission problem by extending the solution by zero in the exterior (resp. interior) domain. Once the unique solvability of this problem has been established, we can then use the results to prove that the same holds for the boundary integral formulation.

Let us first define the space
\[
\mathbb H :=\{ (\boldsymbol{v}, v) \in  {\bf H}^1_{\Delta^*} (\mathbb{R}^d \setminus \Gamma ) \times \mathrm H^1_{\Delta}(\mathbb{R}^d\setminus \Gamma): \jump{\mathcal{R}_{D} ( {\boldsymbol v}, v)} = 0 \}.
\]
The weak formulation of the problem will follow from adding the interior and exterior expressions for the first Green formula  \cref{eq:FirstGreenFormula}, which yields
\begin{equation}\label{eq:JumpFirstGreenFormula}
\mathcal{A}_{\mathbb R^d\setminus\Gamma}\left( \left(\mathbf{u},\theta\right),\left(\boldsymbol{v},v\right)\right) + \mathcal{B}_{\mathbb R^d\setminus\Gamma} \left( \left(\mathbf{u},\theta\right),\left(\boldsymbol{v},v\right)\right)
 = \int_{\Gamma}\jump{\mathcal{R}_N\left(\mathbf{u},\theta\right)}\,\cdot\left(\boldsymbol{v},v\right),
\end{equation}
where we have defined the define the bilinear forms
\[
\mathcal A_{\mathbb R^d\setminus\Gamma}(\cdot,\cdot):=\mathcal A_{\Omega^+}(\cdot,\cdot) + \mathcal A_{\Omega^-}(\cdot,\cdot), \;\;\text{ and }\;\; \mathcal B_{\mathbb R^d\setminus\Gamma}(\cdot,\cdot):=\mathcal B_{\Omega^+}(\cdot,\cdot) + \mathcal B_{\Omega^-}(\cdot,\cdot).
\]
Hence, the weak formulation arises by replacing the term $\jump{\mathcal{R}_N\left(\mathbf{u},\theta\right)}$ by the prescribed transmission conditions, and considering that that if a pair $(\mathbf{u},\theta)$ satisfies the distributional form of the problem associated then $\mathcal{B}_{\mathbb R^d\setminus\Gamma} \left( \left(\mathbf{u},\theta\right),\left(\boldsymbol{v},v\right)\right)=0$.
%
\paragraph{\textbf{Transmission problems}}
%
Now, for given $(\boldsymbol{\lambda}, \varphi) \in \mathbf{ H}^{-1/2}(\Gamma) \times \mathrm H^{1/2}(\Gamma) $,  let  
$\widetilde{ \varphi}  \in  \mathrm H^1(\mathbb{R}^d\setminus \Gamma)$ be an extension of $\varphi$ such that $\mathcal{Q}_{SD} (s)(\bf{0}, \varphi) = ({\bf 0}, \widetilde\varphi) $. Note that, by construction, $(\boldsymbol 0,\widetilde\varphi)$ satisfies \cref{eq:6.3a}  and $\jump{\partial_n \widetilde{\varphi}} = 0$. 

Motivated by \cref{eq:JumpFirstGreenFormula}, we will say that a pair
 $(\mathbf{u}, \theta)  \in \mathbb H$  is a weak solution  of \cref{eq:6.3a}, if it satisfies the variational equation 
\begin{equation} \label{WeakTransmission}
 \mathcal{A}_{\mathbb{R}^d\setminus \Gamma} \Big( ( \mathbf u, \theta), (\boldsymbol v, v) \Big)  = - 
 \mathcal{A}_{\mathbb{R}^d\setminus \Gamma} \Big( ( \boldsymbol 0, \widetilde{\varphi}),( \boldsymbol v, v) \Big) + \int_{\Gamma}(\boldsymbol \lambda, 0)\cdot (\boldsymbol v,v) ,
\end{equation}
for all  test functions $(\boldsymbol v, v) \in  \mathbb H$. The term in the right hand side involving the extension $\widetilde{\boldsymbol\phi}$ accounts for the non-homogeneous jump in the trace. It is clear that if this variational equaton holds, the distributional equation associated to  $\mathcal{B}_{\mathbb R^d\setminus\Gamma} \left(\cdot,\cdot\right)$ holds as well and vice versa. We will now show that $\mathcal A_{\mathbb R^d\setminus\Gamma}(\cdot,\cdot)$ is strongly elliptic. 

If we denote the $d$-dimensional identity operator by $\mathbf I$ and define 
\[
Z(s) := \left( \begin{array} {ll}
 s\,\mathbf{I} & 0 \\
  0 & \gamma/\eta
 \end{array} \right ),
\]
then a simple computation shows that
\begin{align}
\mathrm{Re} \; \mathcal{A}_{\mathbb{R}^d\setminus \Gamma} \left( \left( \mathbf{\overline{u}}, \theta\right) , Z(s) \left(\boldsymbol{u}, \overline{\theta} \right)^{\top}\right) &:=\mathrm{Re} \; 
 \int_{\mathbb{R}^d\setminus \Gamma} s \Big(\bm{\sigma}( \mathbf{ \overline{u}} ) : \bm{ \varepsilon }( \boldsymbol{ u}) + \rho \overline{s}^2 \mathbf{\overline{u}} \cdot \boldsymbol{u} - \gamma\,\overline{\theta}\, \nabla\,\cdot\boldsymbol{u} \Big) dx
 \nonumber \\
&  \quad \qquad + \mathrm{Re} \;  \int_{\mathbb{R}^d\setminus \Gamma}\Big(\,s \gamma  \nabla\cdot\mathbf{u}\, \overline{\theta} + \frac{\gamma}{\eta} ( \nabla \theta \cdot \nabla \overline{\theta}  + \frac{s}{\kappa}
 \theta\,  \overline{\theta} )  \Big)\,  dx  \nonumber\\
 &\ge  \min\{1,\gamma/\eta\}\,\frac{\sigma\underline{\sigma}}{|s|}\,\triple{( \mathbf u, \theta  )}^2_{|s|, \mathbb{R}^d\setminus \Gamma}  \label{eq:EllipticityND}
\end{align}
Therefore, there exists a solution of the variational equation \eqref{WeakTransmission} that we shall denote by $({\bf u}_{\lambda}, \widehat{\theta})$.  Moreover, for this pair it follows from \cref{WeakTransmission}  we see that
\begin{align*} 
 \min\{1,\gamma/\eta\}\,\frac{\sigma\underline{\sigma}}{|s|}\,\triple{( \mathbf u, \theta  )}^2_{|s|, \mathbb{R}^d\setminus \Gamma}  &\leq \mathrm{Re} \;  \mathcal{A}_{\mathbb{R}^d\setminus \Gamma} \Big((\overline{\mathbf u}_{\boldsymbol{\lambda}},\widehat{\theta}), Z(s)(\mathbf u_{\boldsymbol{\lambda}},  \overline{\widehat{\theta}} )^{\top}\Big)\\
 & = - \mathrm{Re} \;\mathcal{A}_{\mathbb{R}^d\setminus \Gamma} \Big( ( \boldsymbol 0, \widetilde{\varphi}), Z(s)(\mathbf u_{\boldsymbol{\lambda}},  \overline{\widehat{\theta}} )^{\top} \Big) 
+ \mathrm{Re} \;  \int_{\Gamma}(\overline{\boldsymbol \lambda}, 0) Z(s)(\mathbf u_{\boldsymbol{\lambda}},  \overline{\widehat{\theta}} )^{\top}  d_{\Gamma}.
\end{align*}
We begin the estimates of the first term on the right-hand side:  
\begin{align}
 - \mathrm{Re} \;\mathcal{A}_{\mathbb{R}^d\setminus \Gamma} \Big( ( \boldsymbol 0, \widetilde{\varphi}), Z(s)(\mathbf u_{\boldsymbol{\lambda}},  \overline{\widehat{\theta}} )^{\top} \Big) 
=\,& \mathrm{Re} \; \int_{\mathbb{R}^d \setminus \Gamma}
 \, s \gamma\, \overline{\widetilde{\varphi}}\,  \nabla \cdot \mathbf {{ u}_{\boldsymbol\lambda} } dx  \, - 
 \mathrm{Re} \; \int_{\mathbb{R}^d \setminus \Gamma} \frac{\gamma}{\eta} \left(  \nabla \widetilde{\varphi} \cdot \overline{\nabla\widehat{ \theta} } +  \frac{s}{\kappa |s|} |s| \widetilde{\varphi}\;  \overline{\widehat{\theta}} 
\, \right) dx \nonumber \\
& \leq | \, \mathrm{Re} \; \; \int_{\mathbb{R}^d \setminus \Gamma} \Big( s \gamma \nabla {\overline{\widetilde{\varphi}} } \cdot 
 \mathbf {{ u}_{\boldsymbol\lambda} } + \frac{\gamma}{\eta} (  \nabla \widetilde{\varphi} \cdot \overline{\nabla\widehat{ \theta} } +  \frac{s}{\kappa |s|} |s| \widetilde{\varphi}\;  \overline{\widehat{\theta}}  ) \Big) dx\,  |  \nonumber \\
 &\leq \gamma\;  max\{ 1, \frac{1}{\eta}\}\sqrt{\frac{|s|}{\underline{\sigma}}} \triple{ \widetilde{\varphi}}_{1, \mathbb{R}^d\setminus \Gamma}\; \triple{( \mathbf u_{\lambda}, \widehat{\theta}  )}_{|s|, \mathbb{R}^d\setminus \Gamma} \label{eq:BoundForA}
\end{align}
where  in the estimates,  an integration by parts  has been tacitly employed  to obtain 
$$ 
 \int_{\mathbb{R}^d \setminus \Gamma}
 \, s \gamma\, \overline{\widetilde{\varphi}}\,  \nabla \cdot \mathbf {{ u}_{\boldsymbol\lambda} } dx 
 = -  \int_{\mathbb{R}^d \setminus \Gamma}  s \gamma \nabla {\overline{\widetilde{\varphi}} } \cdot 
 \mathbf {{ u}_{\boldsymbol\lambda} } dx, 
 $$
 since $\jump{\widehat{\theta} } = {0}$.  Now for the second term,  we we see that 
 \begin{align*}
\left| \mathrm{Re} \;  \int_{\Gamma}(\boldsymbol \lambda, 0) Z(s)(\mathbf u_{\boldsymbol{\lambda}},  \overline{\widehat{\theta}} )^{\top}  d_{\Gamma} \right| &\leq  |s| \,  \|\boldsymbol{\lambda} \|_{ \mathbf H^{-1/2}(\Gamma) }\  \, 
\frac{1}{\underline{\sigma}}\;  \triple{( \mathbf u_{\lambda}, \widehat{\theta}  )}_{|s|, \mathbb{R}^d\setminus \Gamma}
 \end{align*} 
All the constants hidden by the symbol $\lesssim$ depend on the physical parameters $\kappa$, $\rho$, $\eta$, and $\gamma$ but not on the Laplace parameter $s$.  Consequently, combining the above estimate together with  \cref{eq:EllipticityND}, \cref{eq:BoundForA} we obtain 
the inequality
\begin{equation}\label{eq:BoundDNH1}
\triple{( \mathbf u_{\lambda}, \widehat{\theta}  )}_{1, \mathbb{R}^d\setminus \Gamma} \lesssim \frac{|s|^2}{\sigma \underline{\sigma}^3} \left(\triple{ \widetilde{\varphi}}_{1, \mathbb{R}^d\setminus \Gamma}
  + \|\boldsymbol{\lambda} \|_{ \mathbf H^{-1/2}(\Gamma) }\ \right),
\end{equation}
from which the uniqueness of the solution to \eqref{eq:6.3a} follows. If we then define $\theta_{\varphi} := \widehat{\theta} + \widetilde{\varphi}$, it follows that
\begin{align*}
\mathbf{B}(\partial_x , s) (\mathbf u_{\lambda}, {\theta}_{\varphi}) 
& = {\bf 0} \qquad \mbox{in}\quad \mathbb{R}^d \setminus \Gamma,\\
\jump{\mathcal{R}_{DN} ({\bf u}_{\lambda}, \theta_{\varphi}) } & = {\bf 0},\\
\jump{ \mathcal{R}_{ND} ({ \bf u}_{\lambda}, \theta_{\varphi})} & = (\boldsymbol{\lambda}, \varphi).
\end{align*}
 Thus, the pair $(\mathbf u_{\boldsymbol \lambda},\theta_{\varphi})$ is the unique solution of the transmission problem \cref{eq:6.3}.

The proof for the combined Neumann-Dirichlet transmission problem  \eqref{eq:6.9}, follows a very similar argument. For a given pair $(\boldsymbol \phi,\varsigma) \in \mathbf H^{1/2}(\Gamma)\times \mathrm H^{-1/2}(\Gamma)$, we extend $\bm{\phi}$  to  
$\bm{ \widetilde{\phi}}  \in  {\bf H}^1(\mathbb{R}^d\setminus \Gamma)$ with  $(\bm {\widetilde{\phi}}, 0) =  \mathcal{Q}_{DS} (s)(\bm{\phi}, 0)$. Having been defined in terms of layer potentials, the pair $(\widetilde{\boldsymbol\phi},0)$ will satisfy \cref{eq:6.9a} and $\jump{\mathbf T\,\widetilde{\boldsymbol\phi}}=\boldsymbol 0$. We will say that 
 $(\mathbf u , \theta)  \in \mathbb H$ is a weak solution of the transmission problem \cref{eq:6.9} if  it satisfies the variational equation
 \begin{equation}
 \mathcal{A}_{\mathbb{R}^d\setminus \Gamma} \Big((\mathbf u,\theta), (\mathbf v, v) \Big)  = - 
 \mathcal{A}_{\mathbb{R}^d\setminus \Gamma} \Big( (\widetilde{\boldsymbol\phi},0), (\mathbf  v, v) \Big) + \int_{\Gamma} (\boldsymbol 0, \varsigma) \cdot (\mathbf v,v) \label{eq:6.16}
\end{equation}
for all  test functions $({\bf v}, v) \in \mathbb H$. The solvability of the variational formulation follows from \cref{eq:EllipticityND}, therefore there exists at least one pair satisfying \cref{eq:6.9a}, which we shall denote $(\widehat{\mathbf u}, \theta_\varsigma)$,  
for which it follows that 
\begin{align*} 
 \min\{1,\gamma/\eta\}\,\frac{\sigma\underline{\sigma}}{|s|}\,\triple{( \widehat{\mathbf u}, \theta_{\varsigma}  )}^2_{|s|, \mathbb{R}^d\setminus \Gamma}  &\leq \mathrm{Re} \;  \mathcal{A}_{\mathbb{R}^d\setminus \Gamma} \Big( Z(\overline{s})(\widehat{\mathbf u} ,\overline{\theta_{\varsigma}})^{\top}, (\overline{ \widehat{\mathbf u}} , \theta_{\varsigma}) \Big). \\
 & = - \mathrm{Re} \;\mathcal{A}_{\mathbb{R}^d\setminus \Gamma} \Big( Z(\overline{s}) ( \widetilde{\boldsymbol \phi}, 0)^{\top}, (\overline{ \widehat{\mathbf u}}, \theta_{\varsigma} ) \Big) 
+ \mathrm{Re} \;  \int_{\Gamma} Z(\overline{s}) (0,\overline{{\varsigma}})^{\top}\cdot  ( \overline{\widehat{{\mathbf u}}},  \theta_{\varsigma} )  d_{\Gamma}.
\end{align*}
In the same manner, we obtain the estimates of the right hand side  of the above two terms such that 
\begin{subequations}\label{eq:UpperBoundND} 
\begin{align}
 - \mathrm{Re} \;\mathcal{A}_{\mathbb{R}^d\setminus \Gamma} \Big( Z(\overline{s}) ( \widetilde{\boldsymbol \phi}, 0)^{\top}, (\overline{ \widehat{\mathbf u}}, \theta_{\varsigma} ) \Big) 
 =\, & - \mathrm{Re} \int_{\mathbb{R}^d \setminus \Gamma}\!\! \Big( \bar{s} \boldsymbol\sigma ( \widetilde{\boldsymbol{\phi})} : \boldsymbol\varepsilon(\overline{\widehat{\mathbf u}})+ \rho\,  s\,  |s|^2 \widetilde{ \bm{\phi} }\cdot \overline{{\bf \widehat{u}}} + \bar{s}  \gamma \nabla\cdot \overline{\widetilde{\boldsymbol{\phi}}} \, \theta_{\varsigma} \Big) dx   \nonumber \\
 = \, & - \mathrm{Re} \int_{\mathbb{R}^d \setminus \Gamma}\!\! \Big( \bar{s} \boldsymbol\sigma ( \widetilde{\boldsymbol{\phi})} : \boldsymbol\varepsilon(\overline{\widehat{\mathbf u}})+ \rho\,  s\,  |s|^2 \widetilde{ \bm{\phi} }\cdot \overline{{\bf \widehat{u}}} -  \bar{s}  \gamma \nabla \theta_{\varsigma} \cdot \overline{\widetilde{\boldsymbol{\phi}}} \, \Big) dx  \nonumber \\
\leq \, & max\{1, \gamma\} \,  \frac{|s|^2} {\underline{\sigma}^2} \triple{\widetilde{\boldsymbol\phi}}_{1,\mathbb R^d\setminus\Gamma}\; \triple{ (\widehat{\mathbf u}, \theta_{\varsigma})}_{|s|,\mathbb R^d\setminus\Gamma}  
\label{eq:UpperBoundNDa} \\
\mathrm{Re} \;  \int_{\Gamma} Z(\overline{s}) (0,\overline{{\varsigma}})^{\top},  ( \overline{\widehat{{\mathbf u}}},  \theta_{\varsigma} )  d_{\Gamma} \nonumber  \\
\leq\, &  \frac{\gamma}{\eta} \,  \sqrt{\frac{1}{\underline\sigma}} \, \|\varsigma\|_{\mathrm H^{-1/2}(\Gamma)}\triple{\theta_\varsigma}_{|s|,\mathbb R^d\setminus\Gamma} \leq   \frac{\gamma}{\eta} \,  \sqrt{\frac{1}{\underline\sigma}} \, \|\varsigma\|_{\mathrm H^{-1/2}(\Gamma)}\triple{\theta_\varsigma}_{|s|,\mathbb R^d\setminus\Gamma} \label{eq:UpperBoundNDb}
  \end{align}
\end{subequations} 
Again in order  to bound the term $\nabla \cdot \widetilde{\bm{\phi}} $ we have applied the  integration by parts to the corresponding domain integral  in \cref{eq:UpperBoundNDa} and made use of  the jump condition $\jump{\theta_{\varsigma}} = 0 $. Just like before, the constants disregarded by the symbol $\lesssim$ have no dependence on the Laplace parameter $s$.  Combining \cref{eq:UpperBoundNDa} and \cref{eq:UpperBoundNDb} together with the estimate of  ellipticity condition 
$$
 \min\{1,\gamma/\eta\}\,\frac{\sigma\underline{\sigma}}{|s|}\,\triple{( \widehat{\mathbf u}, \theta_{\varsigma}  )}^2_{|s|, \mathbb{R}^d\setminus \Gamma}  \leq \mathrm{Re} \;  \mathcal{A}_{\mathbb{R}^d\setminus \Gamma} \Big( Z(\overline{s})(\widehat{\mathbf u} ,\overline{\theta_{\varsigma}})^{\top}, (\overline{ \widehat{\mathbf u}} , \theta_{\varsigma}) \Big)
$$
yield the inequality 
\begin{equation}\label{eq:BoundNDH1}
\triple{( {\bf \widehat{u}}, \theta_{\varsigma}  )}_{1, \mathbb{R}^d\setminus \Gamma} \lesssim \frac{|s|^3}{\sigma \, \underline\sigma^4}\left(\triple{\widetilde{\boldsymbol\phi}}_{1,\mathbb R^d\setminus\Gamma} + \|\varsigma\|_{\mathrm H^{-1/2}(\Gamma)}\right)
\end{equation}
from which the uniqueness of the solution follows. By defining ${\mathbf u}_{\boldsymbol\phi}:= \widetilde{\boldsymbol\phi} + \widehat{\mathbf u}$ it is clear that $\jump{\mathcal R_{DN}(\mathbf u_{\boldsymbol\phi},\theta_\varsigma)} = (\boldsymbol\phi,\varsigma)$ and  $\jump{\mathcal R_{ND}(\mathbf u_{\boldsymbol\phi},\theta_\varsigma)} = \boldsymbol 0$. Thus, $(\mathbf u_{\boldsymbol\phi},\theta_\varsigma)$ is the unique solution of problem \eqref{eq:6.9}.

\paragraph{\textbf{Boundary integral equations}}

Recall the boundary integral equations \cref{eq:6.5} and \cref{eq:6.11} for the thermoelastic problem
\begin{alignat*}{6}
\mathcal C_{SD} (s)\left(\boldsymbol{\lambda},\varphi\right) :=\,& \mathcal{R}_{DN}\,\mathcal{Q}_{SD}(s)\left(\boldsymbol{\lambda},\varphi\right) (x) =\,&& \left( \mathbf{f},g\right) \quad&\text{(Dirichlet-Neumann)},\\
\mathcal C_{DS}(s) \left(\boldsymbol{\phi},\varsigma\right):=\,& \mathcal{R}_{ND}\,\mathcal{Q}_{DS}(s) \left(\boldsymbol{\phi},\varsigma\right) (x) =\,&& \left( \mathbf{ g},f\right)\qquad& \text{(Neumann-Dirichlet)}. 
\end{alignat*}
Using the results from the previous section we can now show that the combined boundary integral operators $\mathcal C_{SD} (s)$ and $\mathcal C_{DS} (s)$ are in fact invertible. 

\begin{lemma}\label{lem:InverseBounds}
The combined boundary integral operators $\mathcal C_{SD}$ and $\mathcal C_{DS}$ are invertible. Moreover, the following bounds for their inverses hold
\begin{align}
\label{eq:BoundInverseCSD}
\|\mathcal C_{SD}^{-1}(s)\|_{\mathbf H^{1/2}(\Gamma)\times\mathrm H^{-1/2}(\Gamma)\rightarrow \mathbf H^{-1/2}(\Gamma)\times\mathrm H^{1/2}(\Gamma)}\lesssim \frac{|s|^4}{\sigma\underline{\sigma}^5} \\
\label{eq:BoundInverseCDS}
\|\mathcal C_{DS}^{-1}(s)\|_{\mathbf H^{-1/2}(\Gamma)\times\mathrm H^{1/2}(\Gamma)\rightarrow \mathbf H^{1/2}(\Gamma)\times\mathrm H^{-1/2}(\Gamma)} \lesssim \frac{|s|^4}{\sigma\underline{\sigma}^6}
\end{align}
\end{lemma}
\begin{proof}
We begin with the weak formulation of partial differential equations \cref{eq:6.3a} or \cref{eq:6.9a}: 
\begin{align}
 \mathcal{A}_{\mathbb{R}^d\setminus \Gamma}
\left((\mathbf {u}, \; \theta), 
({\mathbf v}, v)\right) & = \int_{\Gamma}  \mathcal{R}_N 
( {\mathbf u}, \; \theta), \; ( \mathbf {v}, \; v)\; 
d\Gamma, \label{eq:6.12}
\end{align}
where $(\mathbf{ u}, \theta ) = ( {\mathbf{u}}_{\boldsymbol\lambda},\theta_\varphi )  $ or  $(\mathbf{ u}, \theta ) = ( {\mathbf{u}}_{\boldsymbol\phi},\theta_\varsigma ) $.
For the Dirichlet-Neumann problem, consider an arbitrary pair $(\boldsymbol\lambda,\varphi)\in\mathbf H^{-1/2}(\Gamma)\times\mathrm H^{1/2}(\Gamma)$ and the layer potential ansatz $(\mathbf u_{\boldsymbol\lambda},\theta_\varphi)= \mathcal Q_{SD}(\boldsymbol\lambda,\varphi)$. By the continuity properties of the double and simple layer potentials we have that 
\[
\left(\mathbf{u}_{\boldsymbol\lambda},\partial_n\theta_\varphi\right)^-=\left(\mathbf{u}_{\boldsymbol\lambda},\partial_n\theta_\varphi\right)^+ \quad \text{ and } \quad \jump{\mathcal R_{ND}(\mathbf u_{\boldsymbol\lambda},\theta_\varphi)} = (\boldsymbol\lambda,\varphi).
\]
Hence, if we use $(\overline{\mathbf {u}}_{\boldsymbol\lambda},\theta_\varphi)$ as test and $(\mathbf u_{\boldsymbol\lambda},\overline{\theta}_\varphi)$ as trial in the boundary term from \cref{eq:6.12}, we observe that 
\begin{align*}
\int_{\Gamma}\!&\mathcal{R}_N\left(\overline{\mathbf{u}}_{\boldsymbol\lambda},\theta_\varphi\right)\,\cdot \left(\mathbf{u}_{\boldsymbol\lambda},\overline{\theta}_\varphi  \right) d_{\Gamma} \\
=\,& \int_\Gamma\!\Big(\!\left(\,  \overline{\mathbf T\,\mathbf u_{\boldsymbol\lambda} - \gamma\theta_\varphi\mathbf n },\; \partial_n\theta_\varphi\right)^- \cdot\left(\mathbf{u}_{\boldsymbol\lambda},\overline{\theta_\varphi}\right)^- - \left(\overline{ \mathbf T\,\mathbf u_{\boldsymbol\lambda} - \gamma\theta_\varphi\mathbf n }, \;  \partial_n\theta_\varphi\right)^+\cdot \left(\mathbf{u}_{\boldsymbol\lambda}, \overline{\theta_\varphi}\right)^+ \!\Big)\;  d_{\Gamma} \\
=\,& \int_\Gamma\!\Big(\!\left(\,  \overline{\mathbf T\,\mathbf u_{\boldsymbol\lambda} - \gamma\theta_\varphi\mathbf n },\; \overline{\theta_{\varphi}}\right)^- \cdot\left(\mathbf{u}_{\boldsymbol\lambda}, \partial_n\theta_\varphi \right)^- - \left(\overline{ \mathbf T\,\mathbf u_{\boldsymbol\lambda} - \gamma\theta_\varphi\mathbf n }, \;  ,\overline{\theta_\varphi} \right)^+\cdot \left(\mathbf{u}_{\boldsymbol\lambda}, \partial_n\theta_\varphi \right)^+ \!\Big) \; d_{\Gamma} \\
=\,& \int_\Gamma\!\jump{\left( \overline{\mathbf T\,\mathbf u_{\boldsymbol\lambda} - \gamma\theta_\varphi\mathbf n},\overline{\theta_\varphi}\right)}\cdot\left(\mathbf{u}_{\boldsymbol\lambda},\partial_n\theta_\varphi\right)  d_{\Gamma}
=\,\int_\Gamma\!\jump{\mathcal R_{ND}(\overline{\mathbf u_{\boldsymbol\lambda}, -\theta_\varphi})}\cdot\mathcal R_{DN}(\mathbf u_{\boldsymbol\lambda},\theta_\varphi)  d_{\Gamma}\\
=\,&\int_\Gamma\!( \overline{\boldsymbol\lambda,\; \varphi})\cdot\mathcal R_{DN}\mathcal Q_{SD}(\boldsymbol\lambda,\varphi)
d_{\Gamma} 
= \,\int_\Gamma\!( \overline{\boldsymbol\lambda,\; \varphi})\cdot\mathcal C_{SD}(\boldsymbol\lambda,\; \varphi). d_{\Gamma}
\end{align*}
Above, it is important to recall that the second entry of the combined boundary operator $\mathcal R_{ND}$ is defined in \cref{eq:RND} with the negative sign, so that
\[
\jump{\left(\overline{\mathbf T\,\mathbf u_{\boldsymbol\lambda} - \gamma\theta_\varphi\mathbf n,\; - \theta_\varphi}\right)} = \jump{\mathcal R_{ND}(\overline{\mathbf u_{\boldsymbol\lambda}, \; \theta_\varphi})} = (\overline{\boldsymbol\lambda,\; \varphi}).
\]
Therefore, going back to \cref{eq:6.12} and recalling the estimate \cref{eq:EllipticityND} it follows that
\begin{align}
\nonumber
\frac{\sigma\underline{\sigma}}{|s|}\,\triple{( \mathbf u_{\boldsymbol\lambda}, \theta_\varphi)}^2_{|s|, \mathbb{R}^d\setminus \Gamma}  \lesssim\,& \mathrm{Re} \;  \mathcal{A}_{\mathbb{R}^d\setminus \Gamma} \Big( \overline{\mathbf (\mathbf u_{\boldsymbol\lambda} },\theta_\varphi),Z(s)(\mathbf u_{\boldsymbol\lambda}, \overline{\theta_\varphi })^\top\Big) \\
\nonumber
=\,& \mathrm{Re}\langle\jump{\mathcal{R}_{ND}\left( \overline{\mathbf{u}_{\boldsymbol\lambda},\theta_\varphi }\right)},Z(s)\left(\mathbf{u}_{\boldsymbol\lambda}, \theta_\varphi \right)^\top\rangle_\Gamma\\
\label{eq:CSDLoweBound}
=\, & \mathrm{Re}\Big\langle (\overline{ \boldsymbol\lambda,\; \varphi} ),  Z(s) \mathcal C_{SD} (\boldsymbol\lambda,\varphi) \Big\rangle_\Gamma, 
\end{align}
which is a fundamental inequality.

For the final step,  it remains to  bound $ \triple{( \mathbf u_{\boldsymbol\lambda}, \theta_\varphi)}^2_{|s|, \mathbb{R}^d\setminus \Gamma}$ below  by $ \|\boldsymbol{\lambda} \|^2_{ \mathbf H^{-1/2} (\Gamma)}$ and $  \| \varphi) \|^2_{ H^{1/2} (\Gamma) }$. We will make use of \cite[estimate 4.12]{HsSa:2020} relating the $\mathrm H^{-1/2}(\Gamma)$-norms of the densities $(\boldsymbol\lambda,\varphi)$ to the energy norm of the functions $(\mathbf u_{\boldsymbol\lambda},\theta_\varsigma) = \mathcal S(s)(\boldsymbol\lambda,\varphi)$, namely
\[
 \|( \boldsymbol{\lambda}, \varsigma) \|^2_{ {\bf H}^{-1/2} (\Gamma) } \lesssim \frac{|s|^2} {{\underline{\sigma}}^3} 
 \triple{( \mathbf u_{\boldsymbol\lambda}, \theta_{\varsigma} )}^2_{|s|, \mathbb{R}^d\setminus \Gamma}.
\]
This inequality implies
\begin{align}
\label{eq:boundLambda}
\|\boldsymbol\lambda\|^2_{\mathbf H^{-1/2}(\Gamma)} =\,& \|(\boldsymbol\lambda,0)\|^2_{\mathbf H^{-1/2}(\Gamma)} \lesssim
 \frac{|s|^2} {{\underline{\sigma}}^3} 
 \triple{( \mathbf u_{\boldsymbol\lambda}, 0 )}^2_{|s|, \mathbb{R}^d\setminus \Gamma} \leq  \frac{|s|^2} {{\underline{\sigma}}^3} 
 \triple{( \mathbf u_{\boldsymbol\lambda}, \theta_{\varphi} )}^2_{|s|, \mathbb{R}^d\setminus \Gamma}, \\
\intertext{and}
\label{eq:boundSigma}
\|\varsigma\|^2_{\mathrm H^{-1/2}(\Gamma)} =\,&\|(\boldsymbol 0,\varsigma)\|^2_{\mathrm H^{-1/2}(\Gamma)}\lesssim
 \frac{|s|^2} {{\underline{\sigma}}^3} 
 \triple{( \boldsymbol 0, \theta_\varsigma)}^2_{|s|, \mathbb{R}^d\setminus \Gamma} \leq  \frac{|s|^2} {{\underline{\sigma}}^3} 
 \triple{( \mathbf u_{\boldsymbol\phi}, \theta_{\varsigma} )}^2_{|s|, \mathbb{R}^d\setminus \Gamma}.
\end{align}
From \cref{eq:boundLambda} combined with \cref{eq:4.6} and the trace theorem it follows that
\begin{align*}
2\frac{|s|^2} {{\underline{\sigma}}^3} 
 \triple{( \mathbf u_{\boldsymbol\lambda}, \theta_{\varphi} )}^2_{|s|, \mathbb{R}^d\setminus \Gamma}\geq\,& \|\boldsymbol\lambda\|^2_{\mathbf H^{-1/2}(\Gamma)} + \frac{|s|^2} {{\underline{\sigma}}^2} 
 \triple{( \mathbf u_{\boldsymbol\lambda}, \theta_{\varphi} )}^2_{1, \mathbb{R}^d\setminus \Gamma} \\
 \geq\,&  \|\boldsymbol\lambda\|^2_{\mathbf H^{-1/2}(\Gamma)} + \frac{|s|^2} {{\underline{\sigma}}^2}  \|\varphi\|^2_{\mathrm H^{1/2}(\Gamma)} \\
 \geq\,& \|\boldsymbol\lambda\|^2_{\mathbf H^{-1/2}(\Gamma)} + \max\{1,|s|^2\}\|\varphi\|^2_{\mathrm H^{1/2}(\Gamma)}\\
 \label{eq:boundarybound}
 \geq\,&\|\boldsymbol\lambda\|^2_{\mathbf H^{-1/2}(\Gamma)} + \|\varphi\|^2_{\mathrm H^{1/2}(\Gamma)}.
\end{align*}
Combining this with \cref{eq:CSDLoweBound} yields
\[
\frac{\sigma\underline{\sigma}^4}{|s|^3}\, \|( \boldsymbol{\lambda}, \varphi) \|^2_{ \mathbf H^{-1/2} (\Gamma)\times\mathrm H^{1/2} (\Gamma) }\lesssim  \mathrm{Re}\Big\langle (\overline{ \boldsymbol\lambda,\; \varphi} ),  Z(s) \mathcal C_{SD} (\boldsymbol\lambda,\varphi) \Big\rangle_\Gamma 
\]
which proves the ellipticity of $\mathcal C_{SD}$ and the bound \cref{eq:BoundInverseCSD} by a Lax-Milgram argument.

A very similar argument can be used for the Neumann-Dirichlet problem, we will give here only the main steps and will skip the details. By taking a layer potential ansatz of the form $(\mathbf u_{\boldsymbol\phi},\theta_\varsigma)= \mathcal Q_{DS}(\boldsymbol\phi,\varsigma)$ for arbitrary densities $(\boldsymbol\phi,\varsigma)\in\mathbf H^{1/2}(\Gamma)\times\mathrm H^{-1/2}(\Gamma)$ and using the continuity properties of the simple and double layer potentials one arrives at
\begin{align*}
\int_{\Gamma}\;   \mathcal{R}_N \left(\mathbf{u}_{\boldsymbol\phi}, \overline{\theta_\varsigma}  \right)  \cdot \left( \overline{\mathbf{u}_{\boldsymbol\phi}} ,\theta_\varsigma \right)  d_{\Gamma} \, & =  \int_{\Gamma}\;  \;   \mathcal{R}_{ND} \mathcal Q_{DS} (\boldsymbol\phi, \varsigma )\; \cdot \jump{\overline{ \mathcal{R}_{DN} ( \mathbf{u}_{\boldsymbol\phi}, \theta_{\varsigma} ) }}  d_{\Gamma} \\
\,& =  \Big\langle\;  \mathcal C_{DS}  (\boldsymbol\phi, \varsigma ),  ( \overline{\boldsymbol\phi, \varsigma } )  \Big \rangle_{\Gamma}
\end{align*}
From here, recalling \cref{eq:6.12} and based on the ellipticity of 
$ \mathrm{Re} \;  \mathcal{A}_{\mathbb{R}^d\setminus \Gamma}( Z(\overline{s} ) (\mathbf u_{\boldsymbol\phi}, \overline{\theta_{\varsigma} } )^{\top}, (\overline{{\mathbf u}}_{\boldsymbol\phi} , \theta_{\varsigma}) )$ we arrive at a similar fundamental equality for the boundary integral operator
 $ \mathcal C_{DS} $:
\begin{equation} \label{eq:CDSCoercivityBound}
\frac{\sigma\underline{\sigma}}{|s|}\,\triple{( \mathbf u_{\boldsymbol\phi}, \theta_\varsigma)}^2_{|s|, \mathbb{R}^d\setminus \Gamma}  \lesssim \mathrm{Re} \langle  Z(\bar{s}) \mathcal C_{DS} (\boldsymbol\phi,\varsigma),\overline{(\boldsymbol\phi,\varsigma)}\rangle_\Gamma.
\end{equation}
In the same manner, from  \cref{eq:boundSigma} combined with  \cref{eq:4.6} and the trace theorem, we obtain 
\begin{align*}
2\frac{|s|^2} {{\underline{\sigma}}^3} 
 \triple{( \mathbf u_{\boldsymbol\phi}, \theta_{\varsigma} )}^2_{|s|, \mathbb{R}^d\setminus \Gamma} 
\geq \, & \underline{\sigma} \; \max\{1,|s|^2\} \; \|\boldsymbol\phi\|^2_{\mathbf H^{1/2}(\Gamma)} + \|\varsigma\|^2_{\mathrm H^{-1/2}(\Gamma)} \\
\geq \, &\underline{\sigma} \left( \|\boldsymbol\phi\|^2_{\mathbf H^{1/2}(\Gamma)} + \|\varsigma\|^2_{\mathrm H^{-1/2}(\Gamma)}\right).
\nonumber 
 \end{align*}
This in turn can be bounded from below in terms of the densities by a combined application of \cref{eq:CDSCoercivityBound} and the above estimate,  leading to
\[
\frac{\sigma\underline{\sigma}^5}{|s|^3}\, \|( \boldsymbol{\phi}, \varsigma) \|^2_{ \mathbf H^{1/2} (\Gamma)\times\mathrm H^{-1/2} (\Gamma) }\lesssim \mathrm{Re}\langle Z(\overline{s})\mathcal C_{DS} (\boldsymbol\phi,\varsigma),\overline{(\boldsymbol\phi,\varsigma)}\rangle_\Gamma.
\]
This establishes the coercivity of $\mathcal C_{DS}$. Finally, an application of the Lax-Milgram lemma leads to \cref{eq:BoundInverseCDS}.
\end{proof}

\section{Results in the time domain}\label{sec:TimeDomain}
%
%
\subsection{Time-domain convolutional boundary integral equations} \label{sec:ConvolutionalEquations}
%
For convenience, we recall the time-domain combined initial boundary value problems for the elastic displacement field ${\bf U}(x,t)$ and temperature filed ${\Theta}(x,t)$ governed by the linear thermo-elasto-dynamic system \cref{eq:GoverningEquations} 
\begin{alignat*}{6}
\rho \frac{\partial^2\mathbf{U}} {\partial t^2} - \Delta^{*} \mathbf{U} + \gamma \, \nabla~\Theta 
 =\,& \mathbf{0} \qquad&& \text{ in } \Omega^\mp\times(0,T), \\ 
 \frac{1}{\kappa}\frac{\partial \Theta}{\partial t} - \Delta \Theta + \eta\; \frac{\partial}{\partial t}(\nabla\cdot \mathbf{U}) =\,& 0 \qquad&& \text{ in } \Omega^\mp\times(0,T), \\
 {\bf U}(x,t) = {\bf 0}, \quad \frac{\partial}{\partial t}{\bf U}(x,t )= {\bf 0}, \quad {\Theta}(x,t) =\,& 0, \quad&& \text{ in } \Omega^\mp\times(-\infty,0),
 \end{alignat*}
together with either Dirichlet-Neumann boundary conditions:
\[
{\bf U}(x,t) =\, {\bf F}(x,t),\quad \mbox{and} \quad \nabla\Theta(x,t)\cdot\mathbf n =\, G(x,t) \qquad \mbox{ on } \quad \Gamma_T,
\]
or Neumann-Dirichlet boundary conditions:
\[
\boldsymbol{\sigma} ( {\bf U}, \Theta ) {\bf n} = {\bf G}(x,t) \quad\mbox{ and }\quad \Theta  =\, F(x,t) \qquad \mbox{on}\quad \Gamma_T.
\]
Let us begin the time-domain analysis with the Dirichlet-Neumann boundary value problem. In the previous section it was shown that for given 
$(\mathbf F (x,\cdot), G(x,\cdot) ) \in \mathbf H^{1/2}(\Gamma) \times \mathrm H^{-1/2}(\Gamma)$, the Laplace-transformed problem has a unique weak  solution $\mathbf H^1( \Omega^\mp) \times \mathrm H^1(\Omega^\mp) \ni (\mathbf{u}_{\bm \lambda}, \theta_{\varphi}) := \mathcal L\{(\mathbf U, \Theta)\}$ in the form 
\[
(\mathbf{u}_{\bm \lambda}, \theta_{\varphi}) =  \mathcal{Q}_{SD} (s)(\boldsymbol{\lambda}, 
 \varphi) \in {\bf H}^1(\Omega^\mp) \times \mathrm H^1(\Omega^\mp),
 \]
 where $(\boldsymbol{\lambda}, \varphi) \in \mathbf H^{1/2}(\Gamma) \times \mathrm H^{-1/2}(\Gamma)$
 is the unique solution of the boundary integral equation of the first kind
 \[
 \mathcal {C}_{SD}(s) ( \boldsymbol{\lambda}, \varphi ) = (\mathbf f(x,s), g(x,s) ) \quad \mbox{on}\quad \Gamma. 
\]
The solutions to the PDE sytem \cref{eq:2.1} and \cref{eq:2.2} are obtained in terms of the boundary data and the layer densities by the two-step process 
\begin{align*}
( \boldsymbol{\lambda}, \varphi ) =\,& \mathcal{C}_{SD}^{-1}(s)( \mathbf f(x,s), g(x,s) ) , \\
(\mathbf{u}_{\bm \lambda}, \theta_{\varphi})  =\,&  \mathcal{Q}_{SD}(s) ( \boldsymbol{\lambda}, \varphi ) = \mathcal{Q}_{SD}(s)\mathcal{C}_{SD}^{-1}(s) (\mathbf f(x,s), g(x,s) ).
\end{align*} 
The time-domain counterparts to these identities are given in terms of the convolutions
\begin{align*}
\mathcal{L}^{-1}\{ ( \boldsymbol{\lambda}, \varphi )\} & = \mathcal{L}^{-1}\{\mathcal{C}_{SD}^{-1}(s)\} \ast 
(\mathbf F(x,t), G(x,t)), \\
 (\mathbf U(x,t), \Theta(x,t) ) &= \mathcal{L}^{-1} \{  \mathcal{Q}_{SD}(s) \circ\mathcal{C}_{SD}^{-1}(s) \} \ast 
(\mathbf F(x,t), G(x,t)).
\end{align*}
Similarly, for the Neumann-Dirichlet problem, the weak solution in the Laplace domain is given in terms of the combined layer potential in the form:
 \[
(\mathbf{u}_{\bm \phi}, \theta_{\varsigma}) =  \mathcal{Q}_{DS} (s)(\boldsymbol{\phi}, 
 \varsigma) \in \mathbf H^1(\Omega^\mp) \times \mathrm H^1(\Omega^\mp).
 \]
As we have shown previously, the densities $( \boldsymbol{\phi}, \varsigma) \in \mathbf H^{1/2}(\Gamma) \times \mathrm H^{-1/2}(\Gamma)$ are the unique solution of the boundary integral equation of the first kind
\[
 \mathcal{C}_{DS}(s) ( \boldsymbol{\phi}, \varsigma ) = (\mathbf g(x,s), f(x,s) ) \quad \mbox{on}\quad \Gamma. 
\]
In the same manner as for the Dirichlet-Neumann problem, we can express the solutions $( \boldsymbol{\phi}, \varsigma )$
 and $(\mathbf{u}_{\bm \phi}, \theta_{\varsigma}) $ in terms of the given boundary data as
\begin{align*}
( \boldsymbol{\phi}, \varsigma ) =\,& \mathcal{C}_{DS}^{-1}(s)( \mathbf g(x,s), f(x,s) ) , \\
(\mathbf{u}_{\bm \phi}, \theta_{\varsigma})  =\,&  \mathcal{Q}_{DS}(s) ( \boldsymbol{\phi}, \varsigma ) = \mathcal{Q}_{DS}(s)\mathcal{C}_{DS}^{-1}(s) (\mathbf g(x,s), f(x,s) ),
\end{align*}
which in the time domain become
\begin{align*}
\mathcal{L}^{-1} \{( \boldsymbol{\phi}, \varsigma )\} & = \mathcal{L}^{-1}\{\mathcal{C}_{DS}^{-1}(s)\} \ast 
(\mathbf{F}(x,t), G(x,t)) , \\
 (\mathbf {U}(x,t), \Theta(x,t) ) &= \mathcal{L}^{-1} \{  \mathcal{Q}_{DS}(s) \circ\mathcal{C}_{DS}^{-1}(s)\} \ast 
(\mathbf{F}(x,t), G(x,t)).  
\end{align*}
%
\subsection{A class of admissible symbols}\label{sec:AdmissibleSymbols}
Having established the coercivity of the combined operators $\mathcal C_{SD}$ and $\mathcal C_{DS}$ in the Laplace domain in \cref{sec:ExistenceAndUniqueness}, we are now in a position to establish the solvability of the time-domain counterpart. In order to state the result that will allow us to transfer our previous analysis in the Laplace domain back in to the time domain via Lubich's method \cite{Lu:1994}, we will first have to define an admissible class of symbols. 

For Banach spaces $X$ and $Y$, we will denote the set of bounded linear operators from $X$ to $Y$ as $\mathcal{B}(X, Y)$ and will say that an analytic function $A : \mathbb{C}_+ \rightarrow \mathcal{B}(X, Y)$ belongs to the class $ \mathcal{A} (\mu, \mathcal{B}(X, Y))$, if there exists  $\mu \in \mathbb{R}$ such that 
\[
\|A(s)\|_{X,Y} \le C_A\left(\mathrm{Re} (s)\right) |s|^{\mu} \quad \mbox{for}\quad s \in \mathbb{C}_+ ,
\]
where $C_A : (0, \infty) \rightarrow (0, \infty) $ is a non-increasing function such that 
\[
C_A(\sigma) \le \frac{ c}{\sigma^m} , \quad \forall \quad \sigma \in ( 0, 1]
\]
for some $m \ge 0$ and constant $c>0$. The reader will notice the resemblance between the admissibility criterion above and the bounds established in \cref{lem:InverseBounds} for the inverses of the operators $\mathcal C_{SD}$ and $\mathcal C_{DS}$. The significance of these bounds will be made clear by the following theorem, that will establish the connection between the Laplace domain operators studied in the previous section, and the solution of the time domain problem under consideration.

\begin{theorem}[See \cite{Sayas:2016}, Proposition 3.2.2 and \cite{Sayas:2016errata}]\label{pr:5.1}
Let $A = \mathcal{L}\{a\} \in \mathcal{A} (k + \alpha, \mathcal{B}(X, Y))$ with $\alpha\in [0, 1)$ and $k$ a non-negative integer.  If $ g \in \mathcal{C}^{k+1}(\mathbb{R}, X)$ is causal and its derivative $g^{(k+2)}$ is integrable, then $a* g \in \mathcal{C}(\mathbb{R}, Y)$ is causal and 
\[
\| (a*g)(t) \|_Y \le 2^{\alpha} C_{\epsilon} (t) C_A (t^{-1}) \int_0^1 \|(\mathcal{P}_2g^{(k)})(\tau) \|_X \; d\tau,
\]
where 
\[C_{\epsilon} (t) := \frac{1}{2\sqrt{\pi}} \frac{\Gamma(\epsilon/2)}{\Gamma\left( (\epsilon+1)/2 \right) } \frac{t^{\epsilon}}{(1+ t)^{\epsilon}}, \qquad (\epsilon :=  1- \alpha \; \;  \mbox{and}\; \;  \mu = k +\alpha)
\]
and 
\[
(\mathcal{P}_2g) (t) =  g + 2\dot{g} + \ddot{g}.
\]
\end{theorem}

%
\subsection{Well posedness in the Time-Domain}\label{sec:TDWellPosendenss}

In order to establish the final time-domain results, we will now make use of the coercivity bounds from \cref{sec:ExistenceAndUniqueness} combined from with an application of \cref{pr:5.1} to the solutions of the combined problems. The time-domain results will be summarized in the following two theorems:
\begin{theorem}[The Dirichlet-Neumann problem]
Consider the vector of boundary data $ \boldsymbol{\mathcal{F}}_{DN} := ({\bf F}(x,t), G(x,t))$. The following two statements hold: \\

{\em (a)} For the densities $( \boldsymbol{\lambda}, \varphi ) \in \mathbf H^{-1/2}(\Gamma) \times \mathrm H^{1/2}(\Gamma)$. \\

If $\boldsymbol{\mathcal{F}}_{DN} \in \mathcal{C}^5(\mathbb{R}, \mathbf H^{1/2} (\Gamma)\times \mathrm H^{-1/2} (\Gamma) ) $ is causal and $\boldsymbol{\mathcal{F}}_{DN}^{(6)} $ is integrable, then $\mathcal{L}^{-1}\{(\boldsymbol{\lambda}, \varphi)\} \in \mathcal{C}(\mathbb{R}, \mathbf H^{-1/2} (\Gamma)\times \mathrm H^{1/2} (\Gamma) )$ is causal and 
\begin{align}
\nonumber
\|\mathcal{L}^{-1}\{( \boldsymbol{\lambda}, \varphi ) \} \|&_{\mathbf H^{-1/2} (\Gamma)\times \mathrm H^{1/2} (\Gamma)} \\
\label{eq:5.17}
&\lesssim\frac{t} {(1+t) } 
\,t\,\max \{1, t^5 \} \int_0^t \|\mathcal{P}_2 \boldsymbol{\mathcal{F}}_{DN}^{(4)}(\tau) \|_ {\mathbf H^{1/2} (\Gamma)\times \mathrm H^{-1/2} (\Gamma) } \,d \tau.
\end{align}

{\em (b)} For the solution $ (\mathbf{U}(x,t), \Theta(x,t) ) = \mathcal{L}^{-1} \{  \mathcal{Q}_{SD}(s) \circ\mathcal{C}_{SD}^{-1}(s)\} \ast \boldsymbol{\mathcal{F}}_{DN}$. \\

If $\boldsymbol{\mathcal{F}}_{DN} \in \mathcal{C}^4(\mathbb{R},\mathbf H^{1/2} (\Gamma)\times \mathrm H^{-1/2} (\Gamma))$ is causal and $\boldsymbol{\mathcal{F}}_{DN} ^{(5)}$ is integrable, then $ ({\bf U}(x,t), {\Theta(x,t)} )\in 
\mathcal{C}(\mathbb{R}, \mathbf H^1(\Omega^\mp)\times\mathrm H^1(\Omega^\mp))$ is causal and 
\begin{align}
\nonumber
\|({\bf U}(\cdot ,t), \Theta(\cdot,t) )\|&_{\mathbf H^1(\Omega^\mp)\times\mathrm H^1(\Omega^\mp)} \\
\label{eq:5.18}
&\lesssim \frac{t}{(1 + t)}\,t\,\max \{1, t^{4\frac{1}{2}} \} \int_0^t \|\mathcal{P}_2 \boldsymbol{\mathcal{F}}_{DN}^{(3)}(\tau) \|_{\mathbf H^{1/2} (\Gamma)\times \mathrm H^{-1/2} (\Gamma) } \,d \tau. 
\end{align}
\end{theorem}
\begin{proof}
The proof of (a) is a direct application of \cref{pr:5.1} and the coercivity estimate \cref{eq:BoundInverseCSD}, where the relevant spaces are $X= \mathbf H^{1/2}(\Gamma)\times\mathrm H^{-1/2}(\Gamma)$, and $Y=\mathbf H^{-1/2}(\Gamma)\times\mathrm H^{1/2}(\Gamma)$. From the estimate $\mathcal{C}_{SD}^{-1}(s)$ given in \cref{lem:InverseBounds} and following the notation introduced in \cref{pr:5.1} it follows that $\alpha = 0, k=3, \epsilon = 1$, and therefore
 \[
 C_{\epsilon} = c\,\frac{t} {(1+t)}, \qquad C_A(t^{-1}) = c\, t\max \{1, t^5\},
 \]
which proves \cref{eq:5.17}. 

For the case (b) we consider the spaces $X= \mathbf H^{-1/2}(\Gamma)\times \mathrm H^{1/2}(\Gamma)$, and  $Y= \mathbf H^{1}(\Omega^\mp)\times \mathrm H^1(\Omega^\mp)$. A direct estimation of the bounds for the composition of operators $\mathcal Q_{SD}(s)\mathcal C^{-1}_{SD}(s)$ of the form $\|\mathcal Q_{SD}(s)\mathcal C^{-1}_{SD}(s)\|\leq\|\mathcal Q_{SD}(s)\|\,\|\mathcal C^{-1}_{SD}(s)\|$ would result in an over estimation and tighter regularity requirements on the problem data. In order to obtain a sharper bound for the operator $\mathcal Q_{SD}(s)\mathcal C^{-1}_{SD}(s)$, we will start from \cref{eq:CSDLoweBound}
\begin{align*}
\triple{( \mathbf u_{\bm \lambda}, \theta_{\varphi}  )}^2_{|s|, \mathbb{R}^d\setminus \Gamma}  \lesssim\,& \frac{|s|}{\sigma\underline\sigma}\mathrm{Re} \;  \mathcal{A}_{\mathbb{R}^d\setminus \Gamma} \Big(\mathbf (\overline{\mathbf u}_{\bm \lambda},\theta_{\varphi}),Z(s)(\mathbf u_{\bm \lambda}, \overline{\theta}_{\varphi})^\top\Big)\\
=\,& \frac{|s|}{\sigma\underline\sigma}\mathrm{Re}\langle (\overline{\boldsymbol\lambda,\varphi}),Z(s)\mathcal C_{SD}(\boldsymbol\lambda,\varphi)\rangle_\Gamma \\
\lesssim\,& \frac{|s|^2}{\sigma\underline \sigma^2}\|(\boldsymbol \lambda,\varphi)\|_{\mathbf H^{-1/2}(\Gamma)\times\mathrm H^{1/2}(\Gamma)}\,\|(\mathbf f,g)\|_{\mathbf H^{1/2}(\Gamma)\times\mathrm H^{-1/2}(\Gamma)}\\
\lesssim\,&\frac{|s|^3} {\sigma\underline{\sigma}^{3\frac{1}{2}}} 
 \triple{( \mathbf u_{\bm \lambda}, \theta_{\varphi} )}_{|s|, \mathbb{R}^d\setminus \Gamma}\,\|(\mathbf f,g)\|_{\mathbf H^{1/2}(\Gamma)\times\mathrm H^{-1/2}(\Gamma)}
\end{align*}
where we have used the fact that $\|Z(s)\|\lesssim |s|/\underline\sigma$, the inequality \cref{eq:boundLambda} and the boundedness of the trace. One further application of \cref{eq:4.6} yields
\[
\triple{( \mathbf u_{\bm \lambda}, \theta_{\varphi}  )}_{1, \mathbb{R}^d\setminus \Gamma} \lesssim  \frac{|s|^3} {\sigma\underline{\sigma}^{4\frac{1}{2}}} 
 \,\|(\mathbf f,g)\|_{\mathbf H^{1/2}(\Gamma)\times\mathrm H^{-1/2}(\Gamma)}.
\]
Hence, 
\[
\|\mathcal Q_{SD}(s)\mathcal C^{-1}_{SD}(s)\|_{\mathbf H^{-1/2}(\Gamma)\times\mathrm H^{1/2}(\Gamma)\rightarrow \mathbf H^{1}(\Omega^\mp)\times\mathrm H^{1}(\Omega^\mp)}\leq \frac{|s|^3} {\sigma\underline{\sigma}^{4\frac{1}{2}}}.
\]
We can now use \cref{pr:5.1} with $\alpha =0 , k=3 , \epsilon = 1$ and
\[
C_\epsilon = C\frac{t}{1+t}, \quad C_A(t^{-1}) = ct\max\{1,t^{4\frac{1}{2}}\},
\]
arriving at \cref{eq:5.18}.
\end{proof}
The corresponding result for the Neumann-Dirichlet problem is stated in the following theorem.
\begin{theorem}[The Neumann - Dirichlet problem] Let $\boldsymbol{\mathcal{F}}_{ND} := ({\bf G}(x,t), F(x,t))$. \\

{\em (a)} For the densities $( \boldsymbol{\phi}, \varsigma) \in \mathbf H^{1/2}(\Gamma) \times \mathrm H^{-1/2}(\Gamma)$. \\

If $ \boldsymbol{\mathcal{F}}_{ND} \in \mathcal{C}^5 ( \mathbb{R}, \mathbf H^{-1/2} (\Gamma)\times \mathrm H^{1/2} (\Gamma) )$ is causal and $\boldsymbol{\mathcal{F}}_{ND}^{(6)} $ is integrable, then $\mathcal{L}^{-1}\{(\boldsymbol{\phi}, \varsigma)\} \in \mathcal{C}( \mathbb{R}, \mathbf H^{1/2} (\Gamma)\times \mathrm H^{-1/2} (\Gamma))$ is causal and 
\begin{align}
\nonumber
\|\mathcal{L}^{-1}\{( \boldsymbol{\phi}, \varsigma ) \} \|&_{\mathbf H^{1/2}(\Gamma)\times\mathrm H^{-1/2}(\Gamma) } \\
\label{eq:5.20}
&\lesssim \frac{t} {(1+ t) } 
 \,t\,\max \{1, t^6 \} \int_0^t \|\mathcal{P}_2 \boldsymbol{\mathcal{F}}_{ND}^{(4)}(\tau) \|_ {\mathbf H^{-1/2} (\Gamma)\times \mathrm H^{1/2} (\Gamma) } \,d \tau.
\end{align}

{\em (b)} For the solution $ ({\bf U}(x,t), \Theta(x,t) ) = \mathcal{L}^{-1} \{  \mathcal{Q}_{DS} \circ\mathcal{C}_{DS}^{-1}(s) \} \ast \boldsymbol{\mathcal F}_{DS}$. \\

If $ \boldsymbol{\mathcal{F}}_{ND} \in \mathcal{C}^{4}( \mathbb{R}, \mathbf H^{-1/2} (\Gamma)\times \mathrm H^{1/2} (\Gamma) )$ is causal and $\boldsymbol{\mathcal{F}}_{ND}^{(5)} $ is integrable, then $({\bf U}(x,t), \Theta(x,t) )\in \mathcal C(\mathbb R, \mathbf{H}^1(\Omega^\mp)\times \mathrm{H}^1(\Omega^\mp))$ is causal and we have the estimate: 
\begin{align}
\nonumber
\|({\bf U}(\cdot,t)\}, {\Theta(\cdot,t)} )\|&_{\mathbf{H}^1(\Omega^\mp)\times \mathrm{H}^1(\Omega^\mp))} \\
\label{eq:5.21}
& \lesssim \frac{t}{(1 + t)} \,t\,\max \{1, t^{5} \} \int_0^t \|\mathcal{P}_2\boldsymbol{ \mathcal{F}}_{ND}^{(3)}(\tau) \|_ {\mathbf H^{-1/2} (\Gamma)\times \mathrm H^{1/2} (\Gamma) } \,d \tau. 
\end{align}
\end{theorem}
\begin{proof}
As in the Dirichlet-Neumann case, the bound \cref{eq:5.20} on (a) follows immediately from the one for $\mathcal{C}_{DS}^{-1}(s)$  given by \cref{eq:BoundInverseCDS} in \cref{lem:InverseBounds}. Then, by identifying the spaces $X = \mathbf H^{-1/2} (\Gamma)\times \mathrm H^{1/2} (\Gamma) $ and $Y = \mathbf H^{1/2} (\Gamma)\times \mathrm H^{-1/2} (\Gamma)$, together with $\alpha = 0, \,k = 3$, and $\epsilon = 1$,  \cref{pr:5.1} can be applied leading to \cref{eq:5.20}.

To prove (b), we first need to derive a tighter bound for the composition of operators  $\mathcal{Q}_{DS}(s) \mathcal{C}_{DS}^{-1}(s)$, as the product of the separate bounds for each of the two operators, would yield an unnecessarily loose estimate. Instead, we recall \cref{eq:CDSCoercivityBound} 
\[
\frac{\sigma\underline{\sigma}}{|s|}\,\triple{( \mathbf u_{\boldsymbol\phi}, \theta_\varsigma)}^2_{|s|, \mathbb{R}^d\setminus \Gamma}  \lesssim \mathrm{Re}\langle Z(\overline{s})\mathcal C_{DS} (\boldsymbol\phi,\varsigma),\overline{(\boldsymbol\phi,\varsigma)}\rangle_\Gamma,
\]
which implies that 
\begin{align}
\nonumber
\frac{\sigma\underline{\sigma}}{|s|} \triple{( \mathbf u_{\phi}, \theta_{\varsigma} )}^2_{|s|, \mathbb{R}^d\setminus \Gamma}
\lesssim\,& \frac{|s|}{\underline\sigma}\|\mathcal{C}_{DS}( \mathbf u_{\boldsymbol\phi}, \theta_{\varsigma} )\|_{ \mathbf H^{-1/2}(\Gamma)\times \mathrm H^{1/2}(\Gamma) }\|(\boldsymbol\phi,\varsigma)\|_{\mathbf H^{1/2}(\Gamma)\times \mathrm H^{-1/2}(\Gamma)} \\
\nonumber
=\,& \frac{|s|}{\underline\sigma}\|( \mathbf g , f ) \|_{\mathbf H^{-1/2}(\Gamma)\times \mathrm H^{1/2}(\Gamma)}\,\|(\boldsymbol\phi,\varsigma)\|_{\mathbf H^{1/2}(\Gamma)\times \mathbf H^{-1/2}(\Gamma)}\\
\nonumber
\lesssim\,& \frac{|s|^2}{\underline\sigma^3}\,\|( \mathbf g , f ) \|_{\mathbf H^{1/2}(\Gamma)\times \mathrm H^{-1/2}(\Gamma)}\,\triple{( \mathbf u_{\phi}, \theta_{\varsigma} )}_{|s|, \mathbb{R}^d\setminus \Gamma},
\end{align}
where $\|Z(s)\|\lesssim |s|/\underline\sigma$ was used in the first inequality and for the last inequality we used \cref{eq:4.6a}, \cref{eq:boundSigma} and the trace theorem. Hence, using \cref{eq:4.6c} we have 
\[
 \triple{( \mathbf u_{\phi}, \theta_{\varsigma} ) }_{1, \mathbb{R}^d\setminus \Gamma} \lesssim \frac{|s|^3}{\sigma\underline{\sigma}^5}
 \|( \mathbf g , f ) \|_{\mathbf H^{1/2}(\Gamma)\times \mathrm H^{-1/2}(\Gamma)},
\]
from which we can conclude that
\begin{equation}\label{eq:5.22}
 \|\mathcal{Q}_{DS}(s) \mathcal{C}_{DS}^{-1}(s)\|_{\mathbf H^{-1/2}(\Gamma)\times \mathrm H^{1/2}(\Gamma) \rightarrow \mathbf H^1(\Omega^\mp)\times\mathrm H^1(\Omega^\mp)}
\leq c \frac{|s|^3}{\sigma\underline{\sigma}^5}. 
\end{equation}
Finally, \cref{eq:5.21} follows from an application of  \cref{pr:5.1} extracting the information $X = \mathbf H^{-1/2}(\Gamma)\times \mathrm H^{1/2}(\Gamma),\, Y = \mathbf H^1(\Omega^\mp)\times\mathrm H^1(\Omega^\mp),\, \alpha = 0,\,k = 3$, and $\epsilon = 1$ from \cref{eq:5.22} above.
\end{proof}

%
\section{Acknowledgements}
%
Tonatiuh S\'anchez-Vizuet was partially funded by the US Department of Energy. Grant No. DE-FG02-86ER53233.

%
\appendix

\section{Fundamental solutions} \label{sec:FundamentalSolutions}
%
For completeness, we present the fundamental solution of eq. \cref{eq:AdjointThermoealsticMatrix} in two and three dimensions. We refer the reader to \cite{HsSa:2020} where a detailed derivation is presented following H\"{o}rmander's method \cite{Ho:1964}.

\paragraph{\textbf{Fundamental solution in 3 dimensions}}
Where the constants $\lambda_1^2, \lambda_2^2, \lambda_3^2$ satisfy the dispersion relations
\begin{alignat}{6}
\lambda_1^2 + \lambda_2^2 &=\frac{s}{\kappa} + \frac{\gamma \,\eta\, s} {\lambda +2 \mu} + \lambda_p^2,  &\qquad \qquad& \lambda_p^2 = \frac{\rho \,s^2} {\lambda + 2 \mu}, \label{eq:A.3}\\
\lambda_1^2 \,\lambda_2^2 &= \frac{s}{\kappa}\, \lambda_p^2,  &\qquad \qquad& \lambda_3^2 = \frac{\rho \,s^2}{\mu}. \nonumber
\end{alignat}
The three-dimensional fundamental solution for the thermoelastic oscillation operator is defined as
\begin{equation}\label{eq:A.11}
\underline {\underline{\bf E}}(x,y;s) = \sum_{k=1}^3 \mathbf{D}_k(x, s) \frac{e^{-\lambda_k |x - y|}}{4 \pi |x - y |}, 
\end{equation}
where $\mathbf{D}_k(x, s)'s $ are matrices of differential operators given by
\begin{eqnarray} 
\mathbf{D}_1(x, s)&:=& \frac{1}{\rho s^2 (\lambda_1^2 - \lambda_2^2)}
 \left (
\begin{array} {l|l} 
\hspace{2mm}\;(\lambda^2_p - \lambda^2_2) \,\nabla \nabla^{\top}
&\hspace{4mm} \gamma \, \lambda^2_p \;\nabla \\ [2mm] \hline\\
\hspace{4mm} s\, \eta \,\lambda^2_p \,\nabla^{\top} & \hspace{4mm} \rho\, s^2\, (\lambda^2_1 - \lambda^2_p) 
\end{array} \right ), \\[4mm]
\mathbf{D}_2(x, s)&:=& \frac{1}{\rho s^2 (\lambda_2^2 - \lambda_1^2)}
 \left (
\begin{array} {l|l} 
\hspace{2mm}\;(\lambda^2_p - \lambda^2_1)\, \nabla \nabla^{\top}
& \hspace{4mm} \gamma \, \lambda^2_p \;\nabla \\ [2mm] \hline\\
\hspace{4mm} s\,\eta \,\lambda^2_p \,\nabla^{\top} &\hspace{4mm} \rho \,s^2\, (\lambda^2_2 - \lambda^2_p) 
\end{array} \right ), \\[1mm]
\mathbf{D}_3(x, s)&:=& \frac{1}{\rho s^2}
 \left (
\begin{array} {l|l} 
\hspace{2mm}\;\lambda^2_3 \;\underline{\underline{\bf I}} - \nabla \nabla^{\top}\hspace{2mm}
&\hspace{4mm} {\mathbf 0} \\ [2mm] \hline\\
{\hspace{8mm}\mathbf 0} &\hspace{4mm} {\mathbf 0}
\end{array} \right ). 
 \end{eqnarray}
%
\paragraph{\textbf{Fundamental solution in two dimensions}}
%
The decomposition of the fundamental solution in the two dimensional case is similar to its three-dimensional counterpart. However, in the 2-D case it is given in terms of modified Bessel functions of the second kind $K_0(\lambda_k|x-y|)$, also known as Macdonald functions. The fundamental solution is given by
\[
\underline {\underline{\bf E}}(x,y) = \sum_{k=1}^3 \mathbf{D}_k(x, s)\frac{1}{2\pi} K_0({\lambda_k |x - y|}), 
\]
where  the matrices of operators $\mathbf{D}_k(x, s) $ are defined as
\begin{eqnarray} 
\mathbf{D}_1(x, s)&:=& \frac{1}{\rho s^2 (\lambda_1^2 - \lambda_2^2)}
 \left (
\begin{array} {l|l} 
\hspace{2mm}\;(\lambda^2_p - \lambda^2_2) \,\nabla \nabla^{\top}
&\hspace{4mm} \gamma \, \lambda^2_p \;\nabla \\ [2mm] \hline\\
\hspace{4mm} s\, \eta \,\lambda^2_p \,\nabla^{\top} & \hspace{4mm} \rho\, s^2\, (\lambda^2_1 - \lambda^2_p) 
\end{array} \right ), \label{eq:A.201} \\[4mm]
\mathbf{D}_2(x, s)&:=& \frac{1}{\rho s^2 (\lambda_2^2 - \lambda_1^2)}
 \left (
\begin{array} {l|l} 
\hspace{2mm}\;(\lambda^2_p - \lambda^2_1)\, \nabla \nabla^{\top}
& \hspace{4mm} \gamma \, \lambda^2_p \;\nabla \\ [2mm] \hline\\
\hspace{4mm} s\,\eta \,\lambda^2_p \,\nabla^{\top} &\hspace{4mm} \rho \,s^2\, (\lambda^2_2 - \lambda^2_p) 
\end{array} \right ), \\[4mm]
\mathbf{D}_3(x, s)&:=&- \frac{1}{\rho s^2}
 \left (
\begin{array} {l|l} 
\hspace{2mm}\;\nabla \nabla^{\top}- \lambda^2_3 \;\underline{\underline{\bf I}}\hspace{2mm}
&\hspace{4mm} {\mathbf 0} \\ [2mm] \hline\\
{\hspace{8mm}\mathbf 0} &\hspace{4mm} {\mathbf 0}
\end{array} \right ),
 \end{eqnarray}
and the constants $\lambda_1^2, \lambda_2^2, \lambda_3^2$ satisfy 
\begin{alignat*}{6}
\lambda_1^2 + \lambda_2^2 =\,& \frac{s}{\kappa} + \frac{\gamma \,\eta\, s} {\lambda +2 \mu} + \lambda_p^2, &\quad \qquad& \lambda_p^2 =\,& \frac{\rho \,s^2} {\lambda + 2 \mu}, \\
\lambda_1^2 \,\lambda_2^2 =\,& \frac{s}{\kappa}\, \lambda_p^2, &\quad \qquad& \lambda_3^2 =\,& \frac{\rho \,s^2}{\mu}.
\end{alignat*}
%
\paragraph{\textbf{Remark.}}
We note that for the adjoint equation, if we let $\underline {\underline{\bf E}}^{*}(x, y;s)$ be the fundamental solution such that 
\begin{equation*}
\mathbf{B}^{*}(\partial_y, s) \underline {\underline{\bf E}}^{*}(x,y;s) = - \delta (y-x) \underline{\underline{\bf I}},
\end{equation*}
then we have 
\[
 \underline {\underline{\bf E}}^{*}(x,y;s) = \underline {\underline{\bf E}}^{\top}( x, y; s), 
\]
where $\underline {\underline{\bf E}}^{\top}( x, y ;s) $ is obtained from $\underline {\underline{\bf E}}(x, y; s)$
by transposing the rows and columns (see \cite{Ku:1979}, p.96, and \cite{HsWe:2008}, p.131). The fundamental solutions given in \cite[p.95]{Ku:1979} (see also 
\cite{DaKo:1988} and \cite{Cakoni:2000}) for the sytstem of time-harmonic oscillation equations can be recovered from the one above by replacing $s$ and $\lambda^2_j $ by $- i~\omega$ and $-\lambda^2_j$, respectively.

\bibliographystyle{siamplain}
\bibliography{References.bib}
\end{document}

%% file: ex_shared.tex
\usepackage[utf8]{inputenc}
\usepackage[english]{babel}
\usepackage{lipsum}
\usepackage{amsfonts}
\usepackage{graphicx}
\usepackage{epstopdf}
\usepackage{algorithmic}
\ifpdf
  \DeclareGraphicsExtensions{.eps,.pdf,.png,.jpg}
\else
  \DeclareGraphicsExtensions{.eps}
\fi
\usepackage{psfrag}
\usepackage{amssymb}
\usepackage[titletoc]{appendix}
\usepackage{amsmath}
\usepackage{mathrsfs}
\usepackage{bm}
\usepackage{changepage}
\usepackage{cleveref}

\newsiamremark{remark}{Remark}
\newsiamremark{hypothesis}{Hypothesis}
\crefname{hypothesis}{Hypothesis}{Hypotheses}
\newsiamthm{claim}{Claim}

\newcommand{\jump}[1]{\left[\!\left[#1\right]\!\right]}
\newcommand{\ave}[1]{\left\{\!\!\left\{#1\right\}\!\!\right\}}
\newcommand{\triple}[1]{|\!|\!|#1|\!|\!|}

\headers{BIE in Transient Thermoelasticity: combined boundary conditions}{G. C. Hsiao and T. S\'anchez-Vizuet}

\title{Boundary integral formulations for transient linear thermoelasticity with combined-type boundary conditions\thanks{
\funding{Tonatiuh S\'anchez-Vizuet was partially funded by the US Department of Energy. Grant No. DE-FG02-86ER53233.}}}

\author{George. C. Hsiao\thanks{Department of Mathematical Sciences, University of Delaware, Newark, DE 19716-2553, USA
  (\email{ghsiao@udel.edu}).}
\and Tonatiuh S\'{a}nchez-Vizuet\thanks{Department of Mathematics, The University of Arizona, Tucson, AZ 85721-0089 USA and Courant Institute of Mathematical Sciences, New York University, New York, NY 10012-1185 USA (\email{tonatiuh@math.arizona.edu}). } }

\makeatletter
\newcommand*{\addFileDependency}[1]{
  \typeout{(#1)}
  \@addtofilelist{#1}
  \IfFileExists{#1}{}{\typeout{No file #1.}}
}
\makeatother
